\documentclass{amsart}

\usepackage{amsmath}
\usepackage{amsthm}
\usepackage{hyperref}
\usepackage{amsfonts,graphics,amsthm,amsfonts,amscd,latexsym}
\usepackage{epsfig}
\usepackage{flafter}
\usepackage{mathtools}
\usepackage{comment}
\usepackage{stmaryrd}

\usepackage{mathabx,epsfig}

\hypersetup{
    colorlinks=true,    
    linkcolor=blue,          
    citecolor=blue,      
    filecolor=blue,      
    urlcolor=blue           
}
\usepackage{tikz}
\usetikzlibrary{graphs,positioning,arrows,shapes.misc,decorations.pathmorphing}

\tikzset{
    >=stealth,
    every picture/.style={thick},
    graphs/every graph/.style={empty nodes},
}

\tikzstyle{vertex}=[
    draw,
    circle,
    fill=black,
    inner sep=1pt,
    minimum width=5pt,
]
\usepackage[position=top]{subfig}
\usepackage{amssymb}
\usepackage{color}

\calclayout

\usetikzlibrary{decorations.pathmorphing}
\tikzstyle{printersafe}=[decoration={snake,amplitude=0pt}]

\newcommand{\WDiv}{\operatorname{WDiv}}

\newcommand{\pp}{\mathbb{P}}

\newcommand{\qq}{\mathbb{Q}}
\newcommand{\zz}{\mathbb{Z}}

\newcommand{\rr}{\mathbb{R}}
\newcommand{\cc}{\mathbb{C}}
\newcommand{\kk}{\mathbb{K}}

\def\O#1.{\mathcal {O}_{#1}}			
\def\pr #1.{\mathbb P^{#1}}				
\def\af #1.{\mathbb A^{#1}}			
\def\ses#1.#2.#3.{0\to #1\to #2\to #3 \to 0}	
\def\xrar#1.{\xrightarrow{#1}}			
\def\K#1.{K_{#1}}						
\def\bA#1.{\mathbf{A}_{#1}}			
\def\bM#1.{\mathbf{M}_{#1}}				
\def\bL#1.{\mathbf{L}_{#1}}				
\def\bB#1.{\mathbf{B}_{#1}}				
\def\bK#1.{\mathbf{K}_{#1}}			
\def\subs#1.{_{#1}}					
\def\sups#1.{^{#1}}

\usepackage{tikz}
\usetikzlibrary{matrix,arrows,decorations.pathmorphing}

\newtheorem{introdef}{Definition}

  \newtheorem{introthm}{Theorem}

  \newtheorem{theorem}{Theorem}[section]
  \newtheorem{lemma}[theorem]{Lemma}
  \newtheorem{proposition}[theorem]{Proposition}
  \newtheorem{corollary}[theorem]{Corollary}

  \newtheorem{definition}[theorem]{Definition}
  \newtheorem{example}[theorem]{Example}

  \newtheorem{question}[theorem]{Question}

\newtheorem{remark}[theorem]{Remark}

\theoremstyle{remark}

\numberwithin{equation}{section}

\usepackage[all]{xy}

\begin{document}

\title[Reductive covers of klt varieties]{Reductive covers of klt varieties}

\thanks{
LB is supported by the Deutsche Forschungsgemeinschaft (DFG) grant BR 6255/2-1. 
}

\author[L.~Braun]{Lukas Braun}
\address{Mathematisches Institut, Albert-Ludwigs-Universit\"at Freiburg, Ernst-Zermelo-Strasse 1, 79104 Freiburg im Breisgau, Germany}
\email{lukas.braun@math.uni-freiburg.de}

\author[J.~Moraga]{Joaqu\'in Moraga}
\address{Department of Mathematics, Princeton University, Fine Hall, Washington Road, Princeton, NJ 08544-1000, USA
}
\email{jmoraga@princeton.edu}

\maketitle

\begin{abstract}
In this article, we study $G$-covers
of klt varieties,
where $G$ is a reductive group.
First, we exhibit an example of a klt singularity admitting a
$\mathbb{P}{\rm GL}_n(\kk)$-cover that is not of klt type. 
Then, we restrict ourselves to $G$-quasi-torsors, a special class of $G$-covers that behave like $G$-torsors outside closed subsets of codimension two. 
Given a $G$-quasi-torsor
$X\rightarrow Y$, 
where $G$ is a finite extension of a torus $\mathbb{T}$,
we show that $X$ is of klt type
if and only if $Y$ is of klt type.
We prove a structural theorem for $\mathbb{T}$-quasi-torsors over normal varieties in terms of Cox rings.
As an application, we show that every sequence of $\mathbb{T}$-quasi-torsors over a variety with klt type singularities is eventually a sequence of $\mathbb{T}$-torsors.
This is the torus version of a result due to Greb-Kebekus-Peternell regarding finite quasi-torsors of varieties with klt type singularities.
On the contrary, we show that in any dimension there exists a sequence of finite quasi-torsors 
and $\mathbb{T}$-quasi-torsors
over a klt type variety,
such that infinitely many of them are not torsors.
We show that every variety with klt type singularities is a quotient of a variety with canonical factorial singularities.
We prove that a variety with Zariski locally toric singularities
is indeed the quotient of a smooth variety by a solvable group. Finally, motivated by the work of Stibitz, we study the optimal class of singularities for which the previous results hold.
\end{abstract}

\setcounter{tocdepth}{1} 
\tableofcontents

\section{Introduction}

In algebraic geometry, we often encounter singularities
which are quotients of other singularities by algebraic groups.
Orbifold singularities are finite quotients of smooth points, 
toric singularities are abelian quotients of smooth points~\cite{CLS11}, and
terminal $3$-fold singularities are finite quotients of 
hypersurface singularities~\cite{Rei87}.
Furthermore,
many interesting
factorial singularities are 
${\rm SL}_n(\kk)$-quotients of smooth points~\cite{Bra21, Bra21b}.
In most cases, the respective group is reductive~\cite{Nag61}. 
Indeed, the reductivity assumption is what 
ensures that the quotient is of finite type. 
Reductive quotients preserve normal singularities
and rational singularities~\cite{Bou78}.
Recently, together with Greb and Langlois,
the authors proved that reductive quotients preserve the singularities
of the minimal model program~\cite{BGLM21}, the so-called klt type singularities~\cite{Kol13}.

In this article, we study a central topic in algebraic geometry:
how to improve the singularities of an algebraic variety
by taking appropriate covers. 
We focus on the singularities of the minimal model program.
To tackle this question, we need to comprehend what type of covers will indeed 
improve the singularities that we are studying.
This will lead us to the concepts of $G$-covers
and $G$-quasi-torsors.

\subsection{G-covers of klt singularities}
Let $G$ be an algebraic group.
A $G$-cover of a singularity $(X;x)$ is an
algebraic singularity $(Y;y)$ endowed with a $G$-action fixing $y$ 
so that $X$ is isomorphic to the quotient 
$Y/\!\!/G$ and $x$ is the image of $y$ (see Definition~\ref{def:G-cover-local}). 
In this setting, we say that $(Y;y)$ is a {\em $G$-cover} of
$(X;x)$ and we say that $(X;x)$ is a {\em $G$-quotient} of $(Y;y)$.
One can think about a $G$-cover of a singularity as a  degenerate principal $G$-bundle
over the singularity having maximal degeneration at the 
distinguished singular point.
$G$-covers often occur in singularity theory;
when replacing a singularity with
its universal cover~\cite{Bra20, BFMS22, LLM19}, 
 when taking the index one cover with respect to a $\qq$-Cartier divisor~\cite{KM98}, and 
when taking the Cox ring of the singularity~\cite{Bra19, BM21}.
$G$-covers are often useful to compute invariants of singularities~\cite{Mor20d}.
Thus, it is natural to ask whether 
a class of singularities is preserved by $G$-covers.
Of course, the answer to this question depends on the choice of $G$.
The first question that we settle on in this article is
whether the class of klt type singularities
is closed under reductive covers. 
Our first theorem is a negative answer to this question. 
We show the existence of a $3$-dimensional toric singularity
admitting a $5$-dimensional $\mathbb{P}{\rm GL}_3(\kk)$-cover which is not of
klt type. 

\begin{introthm}\label{introthrm:gl2-cover}
There exists a $3$-fold toric singularity $(X;x)$
that admits a $\mathbb{P}{\rm GL}_3(\kk)$-cover 
$(Y;y)\rightarrow (X;x)$
from a $5$-dimensional 
singularity $(Y;y)$
which is not of klt type.
\end{introthm}

In Proposition~\ref{prop:gln-cover-free}, 
we give further examples in this direction in which
the group $\mathbb{P}{\rm GL}_n(\kk)$ acts
freely on an open subset $Y$. 
However, the singularities are of higher dimension in these cases.

The previous theorem is the local analog of the well-known fact that
projective bundles over Fano type varieties may not be of Fano type.
Indeed, there exist projective bundles over Fano type surfaces
that are not Mori dream spaces~\cite{GHPS12}. 
However, it is known that split projective bundles over Fano
type varieties are Fano type~\cite{BM21}. 
Furthermore, finite covers of Fano type varieties 
are again Fano type varieties, under a restrictive hypothesis
in case there is ramification in codimension one (see, e.g.,~\cite[Lemma 3.18]{Mor20c}).
These two facts motivate the proof of the following theorem.

\begin{introthm}\label{introthm:torus-finite-cover}
Let $(X;x)$ be a klt type singularity. 
Let $G$ be a finite extension of a torus and 
$Y\rightarrow X$ be a 
$G$-quasi-torsor.
Then $(Y;y)$ is a klt type singularity.
\end{introthm}

A {\em $G$-quasi-torsor} is a special kind of $G$-cover which behaves
like a $G$-torsor outside codimension two subsets of $X$ \emph{and} $Y$. 
$G$-quasi-torsors are also called {\em almost principal fiber bundles} in the literature.
Hence, the class of klt type singularities is preserved
under reductive quotients 
and under $G$-quasi-torsors, whenever $G$ is a finite extension of a torus.
We emphasize that the condition 
on the ramification is necessary: 
even finite covers of a smooth point
with codimension one ramification
may not be of klt type (see Example~\ref{ex:cod-1-ram}).
Note that $G$ being a finite extension of a torus 
is equivalent to asking that the derived subgroup of its connected component is trivial~\cite{Hum75}.
It is an open problem to decide
whether the previous statement holds
for $G$ a reductive group (see Question~\ref{quest:G-qt-klt-type}).
With the previous theorem,
we have found the right type of covers that can improve
our klt type singularity: finite quasi-torsors and torus quasi-torsors. 
Whenever these are torsors, i.e., finite \'etale covers and toric bundles, the class of singularities of our variety will not change.
Thus, we are mostly interested in the finite quasi-torsors and $\mathbb{T}$-quasi-torsors that are not torsors.
These are the covers for which the \'etale class of a singularity may change. 
Moreover, these are exactly the covers detected by the
regional fundamental group of the singularity
and by the local Cox ring of the singularity (see~\cite{Bra20} and Definition~\ref{def:local-Cox-ring}).

\subsection{Torus covers of klt varieties}
As mentioned above,
one way to improve the singularities
of a variety is to produce
$\mathbb{T}$-quasi-torsors.
For instance, all toric varieties
are quotients of a smooth affine variety
by the action of an abelian linear algebraic group.
In a similar vein, the local Cox ring of a singularity often simplifies the singularity.
A natural way to obtain a $\mathbb{T}$-quasi-torsor of a 
variety is to mimic the Cox ring construction.
For example, 
if we consider Weil divisors
$W_1,\dots,W_k$ in $X$
spanning the subgroup $N$ of ${\rm WDiv}(X)$, then we can define the sheaf
\[
\mathcal{R}(X)_N := \bigoplus_{(m_1,\dots,m_k)\in \zz^k}
\mathcal{O}_X(m_1W_1+\dots+m_kW_k).
\] 
Then, the relative spectrum 
\[
Y := {\rm Spec}_X( \mathcal{R}(X)_N) \rightarrow X, 
\]
admits a natural $\mathbb{T}$-cover structure over $X$
which is a $\mathbb{T}$-quasi-torsor.
Here, $\mathbb{T}$ is a $k$-dimensional algebraic torus 
and the action of $\mathbb{T}$ on $Y$ is induced by the $\zz^k$-grading of the sheaf $\mathcal{R}(X)_N$.
Note that $Y\rightarrow X$ is a $\mathbb{T}$-torsor precisely at the points at which all the $W_i$'s are Cartier divisors.
The variety $Y$ will be called a {\em relative Cox space} of $X$.
Indeed, this relative version of the Cox space
locally behaves like the Cox space of the singularities of $X$.

Our next theorem states that every
$\mathbb{T}$-quasi-torsor over a normal variety is equivariantly isomorphic to a relative Cox space.

\begin{introthm}\label{introthm:cox-space-vs-g-covers}
Let $X$ be a normal variety.
Let $Y\rightarrow X$ be a $\mathbb{T}$-quasi-torsor.
Then, we can find Weil divisors $W_1,\dots,W_k$ on $X$
for which there
is a $\mathbb{T}$-equivariant isomorphism
\[
Y \simeq {\rm Spec}_X\left(  \bigoplus_{(m_1,\dots,m_k)\in \zz^k}
\mathcal{O}_X(m_1W_1+\dots+m_kW_k)
\right).
\]
\end{introthm}

Theorem~\ref{introthm:torus-finite-cover} implies that a $\mathbb{T}$-quasi-torsor over a klt type singularity is again of klt type. 
On the other hand, Theorem~\ref{introthm:cox-space-vs-g-covers}, implies that a relative Cox ring of a normal variety
is equivariantly isomorphic to a $\mathbb{T}$-quasi-torsor.
Combining these two results, we obtain that a relative Cox ring of a klt type variety 
also has klt type singularities.

\begin{introthm}\label{introthm-finite-torus-cover-klt-var}
Let $X$ be a variety with klt type singularities.
Let $Y\rightarrow X$ be a relative Cox ring.
Then $Y$ has klt type singularities.
\end{introthm}

In summary, the class of $\mathbb{T}$-quasi-torsors
of klt type varieties agrees with the class of relative Cox spaces.
Furthermore, the Cox spaces again have klt type singularities.
In Example~\ref{ex:sing-improve}, we show that the singularities
can indeed improve by taking the relative Cox ring. 

In~\cite[Theorem 1.1]{GKP16}, the authors show that any sequence of finite quasi-torsors over a variety with klt type singularities
is eventually a sequence of finite torsors.
This means that all but finitely many of the finite quasi-torsors are torsors, i.e., 
finite Galois \'etale covers.
It is natural to ask if a similar principle holds
for $\mathbb{T}$-quasi-torsors.
In this direction, we prove a torus version
of the theorem due to Greb, Kebekus, and Peternell.

\begin{introthm}\label{introthm:iteration-torus-quasi-torsors}
Let $X$ be a variety with klt type singularities.
Consider a sequence of morphisms:
\[
 \xymatrix@R=2em@C=2em{
X=X_0 &
X_1\ar[l]_-{\phi_1} &
X_2\ar[l]_-{\phi_2} &
X_3\ar[l]_-{\phi_3} & 
\dots \ar[l]_-{\phi_4} &
X_i\ar[l]_-{\phi_i} &
X_{i+1}\ar[l]_-{\phi_{i+1}}& 
\dots \ar[l]_-{\phi_{i+2}}&
}
\] 
such that each $\phi_i\colon X_i\rightarrow X_{i-1}$
is a $\mathbb{T}$-quasi-torsor.
Then, there exists $j$ such that
for every $i\geq j$
the morphism $\phi_j$ is a 
$\mathbb{T}$-torsor.
\end{introthm}

\subsection{Iteration of torus and finite covers} 
Our previous theorem states that 
any sequence of $\mathbb{T}$-quasi-torsors 
over a variety with klt type singularities
is eventually a sequence of $\mathbb{T}$-torsors.
It is natural to investigate what happens for sequences of $\mathbb{T}$-quasi-torsors
and finite quasi-torsors, i.e.,
to study mixed sequences of torus and finite covers.
In this direction, we will show that 
all but finitely many of the finite quasi-torsors are indeed torsors.

\begin{introthm}
\label{thm:iteration1}
Let $X$ be a variety with klt type singularities.
Consider a sequence of morphisms:
\[
 \xymatrix@R=2em@C=2em{
X=X_0 &
X_1\ar[l]_-{\phi_1} &
X_2\ar[l]_-{\phi_2} &
X_3\ar[l]_-{\phi_3} & 
\dots \ar[l]_-{\phi_4} &
X_i\ar[l]_-{\phi_i} &
X_{i+1}\ar[l]_-{\phi_{i+1}}& 
\dots \ar[l]_-{i+2}&
}
\]
such that each
$\phi_i$ is either a finite quasi-torsor or a torus quasi-torsor.
Then, all but finitely many of the finite quasi-torsors are torsors, i.e., finite \'etale Galois morphisms.
\end{introthm}

Note that the previous theorem gives a generalization of~\cite[Theorem 1.1]{GKP16}.
Indeed, if we require that each $\phi_i$ is a finite quasi-torsor, then we recover this statement.
It is natural to wonder whether in the context of the previous theorem we can further obtain that all but finitely many of the quasi-torsors are torsors. 
First, notice that such a statement holds trivially for smooth varieties. 
Indeed, by the purity of the branch locus, every finite quasi-torsor over a smooth variety is a finite torsor.
On the other hand, by Theorem~\ref{introthm:cox-space-vs-g-covers}, every $\mathbb{T}$-quasi-torsor over a smooth variety is a $\mathbb{T}$-torsor.
Hence, in order to produce interesting sequences of quasi-torsors we need to consider singular varieties.
Toric singularities are arguably the simplest kind of klt singularities because of their combinatorial nature.
The following theorem shows that even for varieties with toric singularities, we may produce infinite sequences of finite quasi-torsors and $\mathbb{T}$-quasi-torsors
so that infinitely many of the $\mathbb{T}$-quasi-torsors are not torsors.

\begin{introthm}
\label{thm:iteration2}
For each $n\geq 2$,
there exists an $n$-dimensional
projective variety $X^n$
with toric singularities
and an infinite sequence of morphisms:
\[
 \xymatrix@R=2em@C=2em{
X^n=X^n_0 &
X^n_1\ar[l]_-{\phi_1} &
X^n_2\ar[l]_-{\phi_2} &
X^n_3\ar[l]_-{\phi_3} & 
\dots \ar[l]_-{\phi_4} &
X^n_i\ar[l]_-{\phi_i} &
X^n_{i+1}\ar[l]_-{\phi_{i+1}}& 
\dots \ar[l]_-{i+2}&
}
\]
such that the following conditions hold:
\begin{enumerate}
    \item each $\phi_i$ is either a finite quasi-torsor or a $\mathbb{T}$-quasi-torsor, 
    \item infinitely many of the $\phi_i$'s are finite torsors, and 
    \item infinitely many of the $\phi_i$'s are $\mathbb{T}$-quasi-torsors that are not torsors.
\end{enumerate}
\end{introthm}

Note that (2) in the previous theorem is implied by
Theorem~\ref{introthm:iteration-torus-quasi-torsors}.
Thus, the importance  relies on (3).
It shows that a full generalization of Greb-Kebekus-Peternell to the case of $\mathbb{T}$-quasi-torsors and finite quasi-torsors is not feasible.
Our final statement in this subsection says that this failure can be fixed if we restrict ourselves to a special class of $\mathbb{T}$-quasi-torsors.
A $\mathbb{T}$-quasi-torsor $Y\rightarrow X$ is said to be {\em factorial} if the variety $Y$ is factorial.

\begin{introthm}
\label{thm:iteration3}
Let $X$ be a variety with klt type singularities.
Consider a sequence of morphisms
\[
 \xymatrix@R=2em@C=2em{
X=X_0 &
X_1\ar[l]_-{\phi_1} &
X_2\ar[l]_-{\phi_2} &
X_3\ar[l]_-{\phi_3} & 
\dots \ar[l]_-{\phi_4} &
X_i\ar[l]_-{\phi_i} &
X_{i+1}\ar[l]_-{\phi_{i+1}}& 
\dots \ar[l]_-{\phi_{i+2}}&
}
\]
such that each $\phi_i$ is either
a finite quasi-torsor
or a factorial $\mathbb{T}$-quasi-torsor.
Then, all but finitely many of the $\phi_i$ are torsors.
\end{introthm}

\subsection{Factorial models}
In this section, motivated by the previous statement, we study factorial covers of klt varieties. 
In~\cite[Theorem~1.5]{GKP16}, the authors prove that a variety $X$ with klt type singularities admits
a quasi-\'etale finite Galois cover
$Y\rightarrow X$ for which
$\pi_1(Y^{\rm reg})$ is isomorphic to $\pi_1(Y)$.
In particular, every \'etale cover of $Y^{\rm reg}$ extends to an \'etale cover
of $Y$.
Our next aim is to improve this result
by considering both; finite quasi-torsors and torus quasi-torsors. 
By doing so, we can also improve the local class groups of the variety $Y$ obtained by Greb, Kebekus, and Peternell.
We show that any variety with klt type singularities
is a $G$-quotient of a variety with canonical factorial singularities
for which its \'etale fundamental group
agrees with the \'etale fundamental group of its smooth locus.

\begin{introthm}\label{introthm:factorial-cover}
Let $X$ be a variety
with klt type singularities.
Then, there exists a variety $Y$ 
satisfying the following conditions:
\begin{enumerate}
    \item the natural epimorphism
    $\hat{\pi}_1(Y^{\rm reg})\rightarrow \hat{\pi}_1(Y)$ of \'etale fundamental groups
    is an isomorphism, 
    \item for every finite quasi-\'etale morphism $Y'\rightarrow Y$ the variety $Y'$ has canonical factorial singularities,
    \item $Y$ admits the action of a reductive group $G$, 
    \item the group $G$ is the extension of an algebraic torus by a finite solvable group, and
    \item the isomorphism $X\simeq Y/\!\!/G$ holds.
\end{enumerate}
In particular, $Y$ itself has canonical factorial singularities.
\end{introthm}

In general, this factorial variety is highly non-unique.
The previous theorem can be regarded as a generalization of~\cite[Theorem~1.5]{GKP16}. Theorem~\ref{introthm:factorial-cover}.(1) follows from~\cite[Theorem~1.1]{GKP16}.
An equivalent statement of the latter does not hold for combinations
of finite and torus covers (c.f.~Theorem~\ref{thm:iteration2}).
However, the statement is still valid if we restrict ourselves to finite quasi-torsors 
and factorial torus quasi-torsors.
Thus, we may still apply Theorem~\ref{thm:iteration3}.
We also observe that the singularities of the variety $Y$ 
produced in the previous theorem can not be improved
by taking finite quasi-\'etale covers and relative Cox rings.
Indeed, every finite quasi-torsor
or torus quasi-torsor over $Y$ is a torsor.

A classic topic in algebraic geometry 
is deciding when a variety which is locally a quotient
is indeed globally a quotient.
Fulton asked whether varieties with
finite quotient singularities are  finite quotients of   smooth varieties.
In~\cite{EHKV01}, the authors prove that a variety with 
finite quotient singularities is the quotient of  a smooth variety
by a linear algebraic group.
In~\cite{KV04}, it is proved that a variety with finite quotient singularities admits a finite flat surjection from a smooth variety.
In~\cite[Theorem 1.2]{GS15}, the authors show that a variety with finite abelian quotient singularities that is globally the quotient of a smooth variety by a torus is globally the quotient of a smooth variety by a finite group.
In this last paper, the language of stacks and Cox rings is used.
In this direction, we prove the following positive result 
in the case of locally toric singularities.

\begin{introthm}\label{introthm:toric-quot}
Let $X$ be a variety with locally toric singularities. 
Then, $X$ admits a torus quasi-torsor
which is a smooth variety.
In particular, $X$ is the quotient of a smooth variety by the action of a torus.
\end{introthm}

In the previous theorem, 
a singularity $x\in X$ is said to be 
locally toric if 
there exists a toric variety $T$ and a closed invariant point $t\in T$ such that
$X_x\simeq T_t$.
Here, $X_x$ (resp. $T_t$) is the spectrum of the local ring $\mathcal{O}_{X,x}$ (resp. $\mathcal{O}_{T,t}$).
The previous theorem can be regarded as a generalization of the fact that 
a toric variety is the quotient 
of an open subset of $\mathbb{A}^n$ by a torus action.
A point $x\in X$ is said to be {\em formally toric} if
there exists a toric variety $T$
and a closed invariant point $t\in T$
such that 
$\hat{X}_x\simeq \hat{T}_t$.
The statement of the previous theorem does not hold if we replace the condition
on locally toric singularities
with formally toric singularities (see Example~\ref{ex:thm7-analytically-local}).

\subsection{Normal singularities}

Throughout the introduction, 
we focused on varieties with klt type singularities.
In this last part, we discuss what class of singularities is the optimal class for which the previous theorems work.
We recall the following theorem due to Stibitz (see~\cite[Theorem 1]{Sti17}).

\begin{introthm}
Let $X$ be a normal variety. 
The following conditions are equivalent:
\begin{enumerate}
\item Every sequence of finite quasi-torsors 
    \[
 \xymatrix@R=2em@C=2em{
X=X_0 &
X_1\ar[l]_-{\phi_1} &
X_2\ar[l]_-{\phi_2} &
X_3\ar[l]_-{\phi_3} & 
\dots \ar[l]_-{\phi_4} &
X_i\ar[l]_-{\phi_i} &
X_{i+1}\ar[l]_-{\phi_{i+1}}& 
\dots \ar[l]_-{\phi_{i+2}}&
}
\]
is eventually a sequence of torsors.
\item For every point $x\in X$ the image of the homomorphism
$\hat{\pi}_1^{\rm reg}(X;x)\rightarrow \hat{\pi}_1^{\rm reg}(X)$
is finite.
\end{enumerate}
\end{introthm}

The group $\hat{\pi}_1^{\rm reg}(X)$ is the \'etale fundamental group of the smooth locus of $X$. 
On the other hand, $\hat{\pi}_1^{\rm reg}(X;x)$, called \emph{\'etale regional fundamental group}, is the profinite completion of the fundamental group of the smooth locus around the singularity (see, e.g.,~\cite{Bra19}). The regional fundamental group
of klt type singularities is finite, so the previous theorem recovers~\cite[Theorem 1.1]{GKP16}.
Motivated by the previous result, we prove the following theorem regarding $\mathbb{T}$-quasi-torsors.

\begin{introthm}\label{introthm:optimal-T}
Let $X$ be a normal variety. 
The following conditions are equivalent:
\begin{enumerate}
\item Every sequence of $\mathbb{T}$-quasi-torsors 
    \[
 \xymatrix@R=2em@C=2em{
X=X_0 &
X_1\ar[l]_-{\phi_1} &
X_2\ar[l]_-{\phi_2} &
X_3\ar[l]_-{\phi_3} & 
\dots \ar[l]_-{\phi_4} &
X_i\ar[l]_-{\phi_i} &
X_{i+1}\ar[l]_-{\phi_{i+1}}& 
\dots \ar[l]_-{\phi_{i+2}}&
}
\]
is eventually a sequence of torsors.
\item For every point $x\in X$ the group
${\rm Cl}(X;x)$ is finitely generated.
\end{enumerate}
\end{introthm}

Due to Theorem~\ref{thm:iteration2}, we know that the similar statement for finite quasi-torsors and torus quasi-torsors fails, even for toric singularities. 
 However, in view of Theorem~\ref{thm:iteration3}, we can expect a similar statement to hold for finite quasi-torsors and factorial $\mathbb{T}$-quasi-torsors. 
In order to state the following theorem, we need to introduce the concept of {\em partial quasi-\'etale Henselizations}.

\begin{introdef}
{\em 
Let $X$ be an algebraic variety and $x\in X$ be a point. The {\em partial quasi-\'etale Henselization} of $X$ at $x$, denoted by $X^{ph}_x$, is the spectrum of the colimit of all 
quasi-\'etale covers $\mathcal{O}_{X,x}\rightarrow R$ that extend to quasi-\'etale covers of $X$ itself.
}
\end{introdef} 
With the previous definition, we can 
state the theorem that describes the optimal class of singularities
for which every sequence of
finite quasi-torsors and factorial  $\mathbb{T}$-quasi-torsors is eventually \'etale.

\begin{introthm}\label{introthm:optimal-mixed}
Let $X$ be a normal variety. 
The following conditions are equivalent:
\begin{enumerate}
\item Every sequence of
finite quasi-torsors and 
factorial $\mathbb{T}$-quasi-torsors 
    \[
 \xymatrix@R=2em@C=2em{
X=X_0 &
X_1\ar[l]_-{\phi_1} &
X_2\ar[l]_-{\phi_2} &
X_3\ar[l]_-{\phi_3} & 
\dots \ar[l]_-{\phi_4} &
X_i\ar[l]_-{\phi_i} &
X_{i+1}\ar[l]_-{\phi_{i+1}}& 
\dots \ar[l]_-{\phi_{i+2}}&
}
\]
is eventually a sequence of torsors.
\item For every point $x\in X$, the following two conditions are satisfied:
\begin{enumerate}
    \item[(a)] The image $\hat{\pi}_1^{\rm reg}(X;x)\rightarrow \hat{\pi}_1^{\rm reg}(X)$ is finite, and
    \item[(b)] the Class group ${\rm Cl}(X^{ph}_x)$ is finitely generated.
\end{enumerate}
\end{enumerate}
\end{introthm}
The proofs of Theorem~\ref{introthm:optimal-T}
and Theorem~\ref{introthm:optimal-mixed} are quite similar to those of the statements for klt type singularities.
We will prove these statements in Subsection~\ref{subsec:normal-singularities}.

\subsection*{Acknowledgements}
The authors would like to thank Olivier Benoist and Ofer Gabber for helpful discussions.
The authors would like to thank Burt Totaro for 
providing Example~\ref{ex:semisimple}.

\section{Preliminaries}

In this section, 
we recall the definitions
of the singularities
of the minimal model program.
We also recall the definition
of $G$-quotients
and $G$-quasi-torsors,
and prove some 
preliminary results.
We work over an algebraically closed field $\kk$ of characteristic zero.
All the considered varieties
are normal
unless stated otherwise.
A {\em reductive group} $G$ is a linear algebraic group $G$ for which
the unipotent radical is trivial.

\subsection{Singularities of the MMP}
In this subsection, we recall the definitions of the singularities of the MMP.

\begin{definition}
{\em 
A {\em log pair} $(X,\Delta)$ consists of the data of a quasi-projective variety $X$
and an effective $\qq$-divisor
$\Delta$ for which
$K_X+\Delta$ is $\qq$-Cartier.
Let $x\in X$ be a closed point.
We write $(X,\Delta;x)$ for the log pair $(X,\Delta)$ around $x$. 
When we write statements about $(X,\Delta;x)$, we mean that such statement holds for $(X,\Delta)$
on a sufficiently small neighborhood of $x$.
}
\end{definition}

\begin{definition}
{\em 
Let $(X,\Delta)$ be a log pair.
Let $\pi\colon Y\rightarrow X$ be a projective birational morphism 
from a normal quasi-projective variety $Y$.
Let $E\subset Y$ be a prime divisor.
We let $\Delta_Y$ be the strict transform of $\Delta$ on $Y$.
We fix canonical divisors $K_Y$ on $Y$
and $K_X$ on $X$ for which 
$\pi_* K_Y=K_X$.
The {\em log discrepancy} of
$(X,\Delta)$ at $E$,
denoted by $a_E(X,\Delta)$, is the rational number:
\[
1+{\rm coeff}_E(K_Y-\pi^*(K_X+\Delta)).
\] 
Hence, the following equality holds: 
\[
\pi^*(K_X+\Delta) = 
K_Y+\Delta_Y + (1-a_E(X,\Delta))E.
\] 
The log discrepancy $a_E(X,\Delta)$ only depends on $E$ and does not depend on $Y$.
We say that $(X,\Delta)$ is a {\em Kawamata log terminal pair} 
(or {\em klt pair} for short) 
if the inequality
\[
a_E(X,\Delta)>0
\] 
holds for every prime divisor $E$ over $X$.
We say that the pair $(X,\Delta)$
is {\em log canonical} (or {\em lc} for short) if the inequality 
\[
a_E(X,\Delta)\geq 0
\]
holds for every prime divisor $E$ over $X$.
}
\end{definition}

\begin{definition}
{\em 
We say that $(X,\Delta_0)$ is of {\em klt type} 
if there exists a boundary $\Delta\geq \Delta_0$ on $X$ for which the pair $(X,\Delta)$ is klt.
We say that $(X,\Delta_0)$ is of {\em lc type}
if there exists a boundary $\Delta\geq \Delta_0$ on $X$ for which the pair
$(X,\Delta)$ is lc.
}
\end{definition}

The following proposition is proved in~\cite[\S 4]{BGLM21}.

\begin{proposition}\label{prop:klt-etale}
The klt type condition is an \'etale condition.
More precisely, let $X$ be an algebraic variety, if for every point $x\in X$ we can find an \'etale neighborhood $U_x\rightarrow X$ 
and a boundary $\Delta_x$ 
for which $(U_x,\Delta_x)$ is klt, then there exists a boundary $\Delta$ on $X$ for which
$(X,\Delta)$ is klt.
\end{proposition}

\subsection{$G$-quotients and $G$-quasi-torsors}
In this section, we recall the 
definitions of 
$G$-quotients
and $G$-quasi-torsors.

\begin{definition}\label{def:G-cover-global}
{\em
Let $(X,\Delta)$ be a pair.
Let $G$ be a reductive group acting on $(X,\Delta)$. 
Assume that the quotient $Y:=X/\!\!/G$ exists. 
Then, we say that $Y$ is a {\em $G$-quotient} of $X$.
We also say that $X$ is a {\em $G$-cover} of the variety $Y$.
}
\end{definition}

\begin{definition}\label{def:G-cover-local}
{\em 
Let $(X,\Delta;x)$ be a singularity of a pair.
Assume that $X$ is an affine variety.
Let $G$ be a reductive group acting on 
$(X,\Delta)$ and fixing $x$, i.e., 
we have that 
$g^*\Delta =\Delta$ and $g(x)=x$ for each $g\in G$.
Let $(X,\Delta;x) \rightarrow (Y;y)$ be the quotient morphism where $y$ is the image of $x$.
We say that $(X,\Delta;x)\rightarrow (Y;y)$
is a {\em $G$-quotient} around $x$. 
The morphism $X\rightarrow Y$ will be called a {\em $G$-quotient}.
We say that $Y$ is the {\em $G$-quotient} of $X$
and that $X$ is the {\em $G$-cover} of $Y$.
}
\end{definition}

Now, we turn to define better behaved quotients. 
We introduce the concept
of {\em $G$-quasi-torsors}.

\begin{definition}\label{def:quasi-torsor} 
{\em  
Let $(Y,\Delta_Y)$ be a log pair.
Let $X$ be a variety with 
the action of $G$ reductive
for which
$Y\simeq X/\!\!/G$.
We say that the quotient morphism
$\phi\colon X\rightarrow Y$ is a 
{\em $G$-quasi-torsor} for $(Y,\Delta_Y)$ if the following conditions are satisfied:
\begin{enumerate}
    \item there are codimension two open subsets $U_Y \subset Y$
    and $U_X =\phi^{-1}(U_Y)\subset X$ for which 
    \[
    \phi|_{U_X} \colon U_X\rightarrow U_Y
    \] 
    is a $G$-torsor, and 
    \item the global invertible homogeneous functions on $X$ descend to $Y$ via the induced
    homomorphism 
    $\mathcal{O}(X)^G \simeq \mathcal{O}(Y) \hookrightarrow \mathcal{O}(X)$.
\end{enumerate}
}
\end{definition}

In general, the $G$-quotient $Y$ does not come with a naturally defined boundary.
However, in some cases, it is possible to introduce such boundary
and compare the log discrepancies 
on $X$ with those on $Y$.
The following lemma is well-known to the experts
(see, e.g.,~\cite[Proposition 2.11]{Mor21}).

\begin{lemma}\label{lem:klt-type-finite-cover-quotient}
Let $(X;x)$ be a klt type singularity.
Then, the following statements hold:
\begin{enumerate} 
\item 
Let $G$ be a finite group acting on $(X;x)$.
The $G$-quotient $(Y;y)$ is of klt type. 
\item 
Let $G$ be a finite group
and $(Y;y)\rightarrow (X;x)$ be 
a $G$-quasi-torsor.
Then, $(Y;y)$ is of klt type.
\end{enumerate}
\end{lemma}

\begin{definition}\label{def:ab-quasi-torsor}
{\em 
We say that a $G$-quasi-torsor
is an {\em abelian quasi-torsor} if $G$ is an abelian group. 
We say that a $G$-quasi-torsor
is a {\em torus quasi-torsor} if 
$G$ is a torus.
In this case, we also write $\mathbb{T}$-quasi-torsor or
$\mathbb{T}$-torsor.
A quasi-torsor $Y\rightarrow X$ is said to be a {\em factorial} quasi-torsor if $Y$ is factorial. 
}
\end{definition}

The following lemma follows from the definitions. 

\begin{lemma}\label{lem:quot-qetale-torus-cover}
Let $Y\rightarrow X$ be a 
$\mathbb{T}$-quasi-torsor
and $\mathbb{T}_0 \leqslant \mathbb{T}$ be a sub-torus. 
Let $Y\rightarrow Y'$ be the quotient of $Y$
by $\mathbb{T}_0$
and $Y'\rightarrow X$ be the induced morphism.
Then, both $Y\rightarrow Y'$
and $Y'\rightarrow X$ are 
torus quasi-torsors.
\end{lemma}

\subsection{Cox rings} In this subsection, we recall some statements about Cox rings for singularities and pairs.
First, we define the concept of affine local Cox rings.

\begin{definition}\label{def:local-Cox-ring}
{\em 
Let $(X;x)$ be a singularity.
Assume that ${\rm Cl}(X;x)$ is finitely generated.
Let $N\leqslant {\rm WDiv}(X)$ be a free finitely generated subgroup surjecting onto ${\rm Cl}(X;x)$
and $N^0$ be the kernel of the surjection
$\pi\colon N\rightarrow {\rm Cl}(X;x)$.
Consider a group homomorphism
$\chi\colon N^0\rightarrow \mathbb{K}(X)^*$ for which
\[
{\rm div}(\chi(E))=E
\] 
for all $E\in N^0$.
We call such $\chi$ a {\em character}.
Let $\mathcal{S}$ be the sheaf of divisorial algebras associated to $N$
and $\mathcal{I}$ be the ideal subsheaf generated by
sections $1-\chi(E)$ where $E\in N^0$.
Then, we define the {\em affine local Cox ring} of $(X,\Delta)$ at $x$ to be
\[
{\rm Cox}(X;x)^{\rm aff}_{N,\chi} := \bigoplus_{[D]\in {\rm Cl}(X;x)} 
\frac{\bigoplus_{D'\in \pi^{-1}([D])} \mathcal{S}_{D'}(X) }{\mathcal{I}(X)}.
\] 
}
\end{definition}

Now, we define the concept of relative Cox ring for a log pair.

\begin{definition}\label{def:rel-Cox-ring}
{\em 
Let $(X,\Delta)$ be a log pair.
Let $W_1,\dots,W_k$ be orbifold Weil divisors on $(X,\Delta)$.
Let $N$ be the subgroup of ${\rm WDiv}(X)$ spanned by $W_1,\dots, W_k$.
We define the sheaf 
\[
\mathcal{R}(X)_N : =
\bigoplus_{D \in N} \mathcal{O}_X(D) \simeq 
\bigoplus_{(m_1,\dots,m_k)\in \zz^k} \mathcal{O}_X(m_1W_1+\dots+m_kW_k).
\] 
The ring $\mathcal{R}(X)_N$ is called a 
{\em relative Cox ring} of $X$.
The relative spectrum 
\[
Y:= {\rm Spec}_X(\mathcal{R}(X)_N) \rightarrow X,
\]
is called a {\em relative Cox space} of $X$.
We may also call $\mathcal{R}(X)_N$ the relative Cox ring associated to $N$
and $Y$ the {\em relative Cox space}
associated to $N$.
}
\end{definition}

Note that in the definition of the local Cox ring, we 
quotient by a certain ideal $\mathcal{I}(X)$
which comes from a character $\chi$.
However, in our definition of the relative Cox ring we do not perform such a quotient.
Example~\ref{ex:torsion-quotient} shows some pathology that would happen  otherwise. 
The definition of the relative Cox ring 
does not depend on the choice of $W_i$ in its linear  equivalence class.

\begin{lemma}\label{lem:replacement-lin-equiv}
Let $X$ be an algebraic variety.
Let $W_1,\dots,W_k$ be Weil divisors on $X$ spanning $N$ in ${\rm WDiv}(X)$.
For each $i\in \{1,\dots,k\}$, let $W'_i\sim W_i$.
Let $N'$ be the subgroup 
in ${\rm WDiv}(X)$ spanned by the Weil divisors $W'_i$.
Then, we have a $\mathbb{T}$-equivariant isomorphism
\[
\mathcal{R}(X)_N \simeq \mathcal{R}(X)_{N'}.
\]
\end{lemma}

The following is proved in~\cite[Proposition 4.10]{BM21} for the case of klt type singularities.
The general case follows from the theory of polyhedral divisors~\cite{AH06}.
It states that in the affine setting 
a torus quasi-torsor
is the same as a relative Cox space. 

\begin{lemma}\label{lem:G-cover-is-local-Cox}
Let $X$ be a normal affine variety 
and $x\in X$ a closed point.
Let $Y\rightarrow X$ be a 
$\mathbb{T}$-quasi-torsor
over $X$.
Then, up to shrinking $X$ around $x$, we can find a finitely generated subgroup
$N\leqslant {\rm WDiv}(X)$ for
which the isomorphism
\[
Y \simeq {\rm Spec}\left( 
\bigoplus_{D\in N}H^0(X,\mathcal{O}_X(D))
\right) 
\]
holds.
\end{lemma}

Furthermore, the Cox ring 
in the local setting has klt type singularities (see, e.g.,~\cite[Theorem 3.23]{BM21}).

\begin{lemma}\label{lem:klt-type-torus-cover}
Let $X$ be an affine variety and $(X;x)$ be a klt type singularity.
Let $N\leqslant {\rm WDiv}(X,\Delta)$ be a free finitely generated subgroup
and $N^0 := \ker(N\rightarrow {\rm Cl}(X;x))$.
Let $\chi \colon N^0\rightarrow \kk(X)^*$ be a character.
Then, the spectrum of the affine local Cox ring 
\[
{\rm Cox}(X;x)^{\rm aff}_{N,\chi} 
\] 
has klt type singularities.
\end{lemma}

The following lemma will be used in the comparison of quasi-torsors and relative Cox rings.

\begin{lemma}\label{lem:isom-implies-lin-equiv}
Let $W$ and $W'$ be two Weil divisors on a normal variety $X$.
Assume that there is a $\mathbb{G}_m$-equivariant
isomorphism
\[
{\rm Spec}_X\left(\bigoplus_{m\in \zz}\mathcal{O}_X(mW)\right) \simeq 
{\rm Spec}_X \left(\bigoplus_{m\in \zz}\mathcal{O}_X(mW')\right).
\] 
Then, we have that $W\sim W'$ on $X$.
\end{lemma}

\begin{proof}
This follows verbatim from the proof of~\cite[Construction~1.4.1.1]{ADHL15}. Since the conditions there are different from ours (but lead to the same conclusion), we recall the argument. Let $W-W'={\rm div}(f)$ and define a homomorphism
$$
\eta \colon \langle W_\zz \rangle \to \kk(X)^*\qquad \text{ and } \qquad kW \mapsto f^k.
$$
Then we obtain an equivariant isomorphism between the sheaves $\bigoplus_{m\in \zz}\mathcal{O}_X(mW)$ and $\bigoplus_{m\in \zz}\mathcal{O}_X(mW')$ by mapping $f \in \mathcal{O}_X(kW)$ to $\eta(kW) \cdot f \in \mathcal{O}_X(kW')$.
\end{proof}

\section{G-covers of klt type singularities}

In this section, we study $G$-covers of klt type singularities.
As seen in Example~\ref{ex:cod-1-ram}, 
we need to focus on those $G$-covers that
are unramified over codimension one points.
First, we will show that 
semisimple covers of klt type singularities may not be of klt type.
The following is a generalization of Theorem~\ref{introthrm:gl2-cover}
to higher-dimensional toric singularities.

\begin{theorem}\label{thm:gln-cover-general}
For any $n\geq 2$, 
there exists a $(n+1)$-dimensional
toric singularity $(X;x)$ 
that admits a 
$\mathbb{P}{\rm GL}_{r}(\kk)$-cover 
$Y\rightarrow X$, satisfying the following conditions:
\begin{enumerate}
    \item we have $r=3$ if $n\in \{2,3\}$ and $r=n$ otherwise,
    \item the singularity $(Y;y)$ has dimension $(n+r-1)$, and
    \item the singularity $(Y;y)$ is not of klt type.
\end{enumerate}
\end{theorem}

\begin{proof}
First, we choose an appropriate projective toric variety, depending on the number $n$.
If $n=2$, we choose a smooth projective toric surface $T$ of Picard rank $4$ and fix $r=3$.
If $n=3$, we choose a smooth projective toric threefold $T$ of Picard rank $3$
and fix $r=3$.
For $n\geq 4$, we choose a smooth projective toric $(n-1)$-fold $T$
of Picard rank $n$ and we fix $r=n$.
In any of the previous cases,
by~\cite[Theorem 1.1]{GHPS12}, we can find a vector bundle $\mathcal{E}$
of rank $r$ over $T$
such that the Cox ring of 
$\mathbb{P}(\mathcal{E})$ is not finitely generated.
We fix $G:=\mathbb{P}{\rm GL}_r(\kk)$ to be the projective linear group acting on $\mathbb{P}(\mathcal{E})$.
Note that $\pi\colon \mathbb{P}(\mathcal{E})\rightarrow T$ is a quotient for the $G$-action.
Let $A_T$ be an ample toric divisor on $T$. 
Let $m$ be a positive integer, 
\[ 
\mathcal{O}_{\mathbb{P}(\mathcal{E})}(1)\otimes \mathcal{O}_{\mathbb{P}(\mathcal{E})}(\pi^*A_T/m) 
\] 
is an ample $\qq$-line bundle on
$\pp(\mathcal{E})$, for $m$ large enough, which is $G$-invariant.
Thus, the affine variety 
\[
Y = {\rm Spec}
\left( 
\bigoplus_{n\in \zz} 
H^0\left(
\mathbb{P}(\mathcal{E}), 
\mathcal{O}_{\mathbb{P}(\mathcal{E})}(n)\otimes
\mathcal{O}_{\mathbb{P}(\mathcal{E})}(n\pi^*A_T/m)
\right) 
\right) 
\] 
admits a $G$-action which fixes the vertex $y\in Y$ of the $\mathbb{G}_m$-action induced by the $\mathbb{Z}$-grading.
Observe that an element 
\[
f\in H^0\left(\mathbb{P}(\mathcal{E}), 
\mathcal{O}_{\mathbb{P}(\mathcal{E})}(n)
\otimes
\mathcal{O}_{\mathbb{P}(\mathcal{E})}(n\pi^*A_T/m)\right)
\] 
is preserved by the action of $G$
if and only if it is constant along the fibers
of $\pp(\mathcal{E})\rightarrow T$.
In other words, the $G$-invariant elements 
have the form 
$f=\pi^*g$ for some
$g\in H^0(T,\mathcal{O}_T(nA_T/m))$.
We conclude that there is an isomorphism 
\[
Y/\!\!/G \simeq 
{\rm Spec}\left( 
\bigoplus_{n\in \zz}H^0(T,\mathcal{O}_T(nA_T/m))
\right)=X.
\] 
Thus, the quotient $Y/\!\!/G$ is isomorphic to the cone over a $\qq$-ample toric divisor on a
smooth projective toric variety. 
Hence, $(X;x)$ is a toric singularity 
of dimension $n$.
It suffices to check
that $(Y;y)$ is not of klt type. 

We proceed by contradiction.
Assume that $(Y;y)$ is of klt type. 
Let $\widetilde{Y}\rightarrow Y$ be the blow-up of $Y$ at the maximal ideal of $y$. Then, the exceptional divisor $E$ of $\phi\colon \widetilde{Y}\rightarrow Y$ is isomorphic
to $\pp(\mathcal{E})$. 
Since $\pp(\mathcal{E})$ is smooth, we conclude that $\widetilde{Y}$ has $\qq$-factorial singularities.
Let $\Delta_Y$ be the effective divisor
through $y$ for which $(Y,\Delta_Y;y)$ has klt singularities.
Let $\Delta_{\widetilde{Y}}$ be the strict transform of $\Delta_Y$ on $\widetilde{Y}$.
We write
\[
\phi^*(K_Y+\Delta_Y)=
K_{\widetilde{Y}}+\Delta_{\widetilde{Y}}+(1-a)E,
\]
for some positive number $a$.
Note that $\Delta_{\widetilde{Y}}$ is ample over $Y$ as $\rho(\widetilde{Y}/Y)=1$.
We conclude that $K_{\widetilde{Y}}+(1-a)E$ is antiample over $Y$, 
so $K_Y+E$ is antiample over $Y$ as well.
Since $E$ is smooth, we conclude that
the pair $(Y,E)$ is plt.
Thus, the pair $K_E+\Delta_E=(K_Y+E)|_E$, obtained by performing adjunction to $E$, is log Fano.
In particular, the projective variety
$\mathbb{P}(\mathcal{E})\simeq E$ is of Fano type.
Thus, the Cox ring of $\mathbb{P}(\mathcal{E})$ is finitely generated by~\cite[Corollary 1.9]{BCHM10}.
This leads to a contradiction.
We conclude that $(Y;y)$
is not a klt type singularity.
\end{proof}

In the previous theorem, the action is not free outside the point $y\in Y$.
We show that this can be improved in the following statement.

\begin{proposition}\label{prop:gln-cover-free}
For any $n\geq 2$, there exists a $(n+1)$-dimensional toric singularity 
$(X;x)$ that admits a 
$\mathbb{P}{\rm GL}_r(\kk)$-cover 
$Y\rightarrow X$, satisfying the following conditions:
\begin{enumerate}
    \item we have that $r=3$ if $n\in \{2,3\}$ and $r=n$ otherwise,  
    \item the germ $(Y;y)$ has dimension $n+r^2$, 
    \item the action of
    $\mathbb{P}{\rm GL}_r(\kk)$ on $Y$
    is free on a dense open set, and
    \item the singularity $(Y;y)$ is not of klt type.
\end{enumerate}
\end{proposition}

\begin{proof}
Let $T$ be a $n$-dimensional smooth projective toric variety.
Let $\mathbb{P}(\mathcal{E})$ be the rank $r$ vector bundle over $T$
considered in the proof of Theorem~\ref{thm:gln-cover-general}.
Hence, the variety $\mathbb{P}(\mathcal{E})$ is not a Mori dream space.
Let $Y_0\rightarrow T$ be the associated principal
$\mathbb{P}{\rm GL}_r(\kk)$-bundle.
Consider a $\mathbb{P}{\rm GL}_r(\kk)$-equivariant projectivization
$Y_0\hookrightarrow \bar{Y}$
with a relatively ample line bundle
$\mathcal{O}_{\bar{Y}}(1)$ over $T$.
Observe that the action of $\mathbb{P}{\rm GL}_r(\kk)$ on 
$\bar{Y}$ is free on the open subset $Y_0$.
We claim that $\bar{Y}$ is not a Mori dream space. 
Let $\rho\colon \mathbb{P}{\rm GL}_r(\kk)\rightarrow {\rm Aut}(\pp^r)$
be the standard representation. 
Then $\bar{Y}\times \pp^r$ admits the action of
$G$ given by $g\cdot (u,v)=(g^{-1}u,\rho(g)v)$.
The quotient of $\bar{Y}\times \pp^r$ by $G$ is isomorphic $\mathbb{P}(\mathcal{E})$.
If $\bar{Y}$ is a Mori dream space, then $\mathbb{P}(\mathcal{E})$ is also a Mori dream space by~\cite[Theorem 1.1]{Oka16}.
This leads to a contradiction.
We conclude that $\bar{Y}$ is not a Mori dream space.
The rest of the proof proceeds as in Theorem~\ref{thm:gln-cover-general},
by replacing 
$\mathbb{P}(\mathcal{E})$
with $\bar{Y}$.
\end{proof}

Now, we turn to prove that 
$G$-covers of klt type singularities
are again of klt type,
provided that $G$ is a finite extension of a torus. 
The following is a generalization of Theorem~\ref{introthm:torus-finite-cover}
which allows ramification over codimension one points. 

\begin{theorem}\label{thm:torus-finite-cover-klt-sing}
Let $(X,\Delta;x)$ be a klt type singularity.
Let $G$ be a finite extension of a torus.
Let $\pi\colon Y\rightarrow X$ be a
$G$-quasi-torsor.
Then, the variety $Y$ is of klt type.
\end{theorem}

\begin{proof}
By Lemma~\ref{lem:klt-type-finite-cover-quotient}, we know that klt type singularities
are preserved under finite covers
and finite quotients.
Hence, we may assume that $G\simeq \mathbb{G}_m^k$ for some $k$.
By Lemma~\ref{lem:G-cover-is-local-Cox}, we know that there exists a 
finitely generated subgroup
$N\leqslant {\rm WDiv}(X,\Delta)$
such that the isomorphism
\[
Y \simeq {\rm Spec}\left( \bigoplus_{D\in N} H^0(X,\mathcal{O}_X(D)) \right)
\]
holds.
Let $\pi\colon N\rightarrow {\rm Cl}(X,\Delta;x)$ be the induced homomorphism. 
Let $N_0$ be the kernel of $\pi$.
We can choose a homomorphism
$\chi\colon N_0\rightarrow \kk(X)^*$
so that
${\rm div}(\chi(E))=E$ for every $E\in N_0$.
Then, we can define the local-affine Cox ring of $(X,\Delta;x)$ associated to the data
$N,\chi$ as in Definition~\ref{def:local-Cox-ring}.
We denote this ring by
\[
{\rm Cox}(X,\Delta;x)^{\rm aff}_{N,\chi}
\] 
and 
we denote by $Y_0$ its spectrum.
By Lemma~\ref{lem:klt-type-torus-cover}, we know that $Y_0$ has klt type singularities.
Applying Lemma~\ref{lem:G-cover-is-local-Cox}
 to the torus cover 
$Y\rightarrow Y_0$, 
we can find a free finitely generated subgroup $N_1\leqslant {\rm CaDiv}(Y_0)$ 
for which the isomorphism 
\[
Y \simeq {\rm Spec}\left( 
\bigoplus_{D\in N_1} H^0(Y_0,\mathcal{O}_{Y_0}(D))
\right) 
\]
holds. 
Observe that we can choose the divisors of $N_1$ to be Cartier on $Y_0$.
Indeed, these divisors correspond to the divisors of $N$, which become Cartier on $Y_0$.
Thus, we conclude that the 
torus quotient 
$Y\rightarrow Y_0$
is a principal torus cover.
Hence, the variety $Y$ has klt type singularities, since the klt type property
is locally \'etale  by Proposition~\ref{prop:klt-etale}.
\end{proof}

\section{Weil divisors modulo Cartier divisors}

In this section, we study the group 
of Weil divisors modulo Cartier divisors, which is called the \emph{local class group} in~\cite{BdFFU15}.
In general, the group 
${\rm WDiv}(X)/{\rm CaDiv}(X)$ is not finitely generated.
This group is trivial if and only if $X$ is locally factorial.
In~\cite{BGS11}, the authors prove that the 
$\qq$-factorial and factorial locus of an algebraic variety
are open.
In~\cite[Section 14]{Kol20}, the author studies the non-$\qq$-Cartier loci
of Weil divisors.
We recall the following proposition due to Koll\'ar (see~\cite[Proposition 138]{Kol20}).

\begin{proposition}\label{prop:kollar-Cartier}
Let $X$ be a normal proper variety.
Let $Z\subset X$ be an irreducible variety.
There exists a dense open subset $Z^0\subset Z$ such that the following holds.
Let $D$ be a Weil divisor that is Cartier at the generic point $\eta_Z$ of $Z$. 
Then, the divisor $D$ is Cartier at every closed point of $Z^0$.
\end{proposition}

Due to the previous proposition, 
we can prove the following theorem 
using Noetherian induction.

\begin{theorem}\label{thm:wdiv-cdiv}
Let $X$ be a normal variety. 
There are finitely many closed points
$x_1,\dots,x_r\in X$ such that the homomorphism 
\[
{\rm WDiv}(X)/{\rm CaDiv}(X) \rightarrow \bigoplus_{i=1}^r {\rm Cl}(X_{x_i}),
\]
is a monomorphism.
\end{theorem}

\begin{proof}
Let $U_1,\dots,U_s$ be an affine open cover of $X$.
Observe that the homomorphism
\[
{\rm WDiv}(X)/{\rm CaDiv}(X)\rightarrow 
\bigoplus_{i=1}^s {\rm WDiv}(U_i)/{\rm CaDiv}(U_i),
\]
induced by restricting
$D\mapsto (D|_{U_1},\dots,D|_{U_s})$, 
is a monomorphism.
Hence, it suffices to prove the statement for an affine variety.

Without loss of generality, we may assume that $X$ is
affine. Let $\bar{X}$ be its closure in a projective space.
By Proposition~\ref{prop:kollar-Cartier}, there exists an open set $\bar{X}^0\subset \bar{X}$ 
so that every Weil divisor on $\bar{X}$ is Cartier
at every closed point of $\bar{X}^0$.
Let $Z_1,\dots,Z_k$ be the irreducible components of $\bar{X}\setminus \bar{X}^0$.
For each $i\in \{1,\dots,k\}$, we choose
$Z_i^0$ as in the statement of Proposition~\ref{prop:kollar-Cartier}.
Then, we proceed inductively with the 
irreducible components
of each $Z_i\setminus Z_i^0$. 

We obtain a finite set of irreducible subvarieties 
$Z_1,\dots,Z_{r_0} \subset \bar{X}$ 
and dense open subsets 
$Z_i^0 \subset Z_i$
so that the following set-theoretic equality holds
\[
\bar{X} = \bigcup_{i=1}^{r_0} Z_i^0.
\]
We may assume that there exists $r\leq r_0$ 
for which 
\begin{equation}\label{eq:covering-affine}
X=\bigcup_{i=1}^r Z_i^0\cap X,
\end{equation}
and each intersection $Z_i^0\cap X$ is non-empty
for $i\in \{1,\dots,r\}$.
For each $i\in \{1,\dots,r\}$, we choose a closed point $x_i \in Z_i^0 \cap X$.
The homomorphism 
\begin{equation}\label{eq:hom-cl} 
{\rm WDiv}(X)/{\rm CaDiv}(X) \rightarrow \bigoplus_{i=1}^r {\rm Cl}(X_{x_i}),
\end{equation} 
is well-defined. Indeed, Cartier divisors are mapped to the zero element on the right-hand side.

It suffices to prove that~\eqref{eq:hom-cl} is a monomorphism. 
Let $D$ be a Weil divisor on $X$.
Let $\bar{D}$ be the closure of $D$ on $\bar{X}$.
Assume that $[D_{x_i}]=0 \in {\rm Cl}(X_{x_i})$ for every $i\in \{1,\dots,r\}$.
Then, $\bar{D}$ is Cartier at the generic point 
$\eta_{Z_i}$ of $Z_i$ for every $i\in \{1,\dots,r\}$. 
By Proposition~\ref{prop:kollar-Cartier}, we conclude that $\bar{D}$ is Cartier at every closed point of 
$Z_i^0$ for every $i\in \{1,\dots,r\}$. 
In particular, $D=\bar{D}\cap X$ is Cartier
at every closed point of
$Z_i^0\cap X$. 
By equality~\eqref{eq:covering-affine}, we conclude that 
$D$ is Cartier at every closed point of $X$.
This means that 
$[D]=0\in {\rm WDiv}(X)/{\rm CaDiv}(X)$.
This finishes the proof of the theorem.
\end{proof}

If the variety has rational singularities,
then we conclude that the group
${\rm WDiv}(X)/{\rm CaDiv}(X)$ is finitely generated. 
In particular, we have the following statement.

\begin{theorem}\label{thm:wdiv-cdiv-klt-type}
Let $X$ be a variety with klt type singularities. 
Then, the group ${\rm WDiv}(X)/{\rm CaDiv}(X)$ is finitely generated.
\end{theorem}

\begin{proof}
Let $X$ be a variety with klt type singularities. Let $x_1,\dots,x_r\in X$ be closed points. 
By~\cite[Theorem 3.27]{BM21}, we know that 
$\bigoplus_{i=1}^r {\rm Cl}(X_{x_i})$ is a finitely generated abelian group.
By Theorem~\ref{thm:wdiv-cdiv}, we conclude that ${\rm WDiv}(X)/{\rm CaDiv}(X)$ is a finitely generated abelian group.
\end{proof}

We have the following corollary from the previous theorem.

\begin{corollary}\label{cor:wdiv-cdiv-pairs}
Let $(X,\Delta)$ be a klt type pair.
Then, the group
${\rm WDiv}(X,\Delta)/{\rm CaDiv}(X)$ is finitely generated.
\end{corollary}

\section{Torus covers of klt type varieties}

In this section, we study torus covers.
We establish a characterization theorem
for torus quasi-torsors over varieties with klt type singularities. 
We will start with the following lemma that will be used in this section.

\begin{lemma}\label{lem:torus-inv-lin-equiv-div}
Let $X$ be a variety that admits a $\mathbb{T}$-action.
Let $W$ be a Weil divisor on $X$.
We can find a $\mathbb{T}$-invariant 
Weil divisor $W'$ on $X$ for which 
$W\sim W'$.
\end{lemma}

\begin{proof}
By Sumihiro's equivariant completion~\cite{Sum74}, we may assume that $X$ is an affine $\mathbb{T}$-variety, where $\mathbb{T}$ is an $n$-dimensional torus.
By~\cite[Proposition 1.6]{AH06}, we can find $\mathbb{T}$-invariant divisors 
$D_1,\dots,D_k$ such that 
$X\setminus \bigcup_{i=1}^k D_i \simeq \mathbb{T}\times U$
for some variety $U$.
We may assume that $U$ is smooth. Let $U\hookrightarrow Y$ be a smooth
projectivization. By~\cite[Exercise 12.6.(b)]{Har77}, we know that
${\rm Cl}(Y\times \mathbb{P}^n)\simeq {\rm Cl}(Y)\times \zz$.
The class group of $Y\times \mathbb{P}^n$ is generated by
$Y\times H$ and 
divisors
of the form $\Gamma_1 \times \mathbb{P}^n,\dots, \Gamma_r\times\mathbb{P}^n$,
where $\Gamma_i \subset Y$ are prime divisors
and $H\subset \pp^n$ is a hyperplane.
Hence, the Class group of $\mathbb{T}\times U$ is generated by torus invariant divisors. 
This implies that the Class group of $X$ is generated by torus invariant divisors.
\end{proof}

Now, we can prove that every
torus quasi-torsor is a relative Cox ring.

\begin{proof}[Proof of Theorem~\ref{introthm:cox-space-vs-g-covers}]
Let $X$ be a normal variety.
Let $Y\rightarrow X$ be a $\mathbb{T}$-quasi-torsor. 
We will prove the statement by induction on $k$ the dimension of $\mathbb{T}$.
First, we will show that the statement holds for $k=1$.

Let $\pi\colon Y\rightarrow X$ be a
$\mathbb{G}_m$-quasi-torsor. 
Let $U\subset X$ be the largest open subset of $X$
for which there is a $\mathbb{G}_m$-equivariant 
isomorphism 
\begin{equation}\label{eq:local-Cox-U}
\pi^{-1}(U) \simeq {\rm Spec}_U\left(
\bigoplus_{m\in \zz} \mathcal{O}_U(mW)
\right) 
\end{equation} 
for a certain Weil divisor $W$ on $U$.
By Lemma~\ref{lem:G-cover-is-local-Cox}, 
we know that $U$ is not empty. 
We claim that $U=X$. 
By contradiction, assume that 
$U\subsetneq X$. 
Let $x\in X$ be a closed point
contained in the complement of $U$. 
By definition, we can find a Weil divisor $W$ on $X$
for which the isomorphism~\eqref{eq:local-Cox-U} holds.
By Lemma~\ref{lem:G-cover-is-local-Cox}, 
we may find an affine neighborhood $V$ of $X$
and a Weil divisor $W'$ on $V$ for which there
is a $\mathbb{G}_m$-equivariant isomorphism 
\begin{equation}\label{eq:isom-V}
\pi^{-1}(V) \simeq {\rm Spec}_V\left( 
\bigoplus_{m\in \zz} \mathcal{O}_V(W')
\right).
\end{equation} 
In particular, there is a $\mathbb{G}_m$-equivariant
isomorphism 
\[ 
{\rm Spec}_{U\cap V}\left( 
\bigoplus_{m\in \zz} \mathcal{O}_{U\cap V}(mW|_{U\cap V)}
\right) 
\simeq 
{\rm Spec}_{U\cap V}\left( 
\bigoplus_{m\in \zz} \mathcal{O}_{U\cap V}(mW'|_{U\cap V)}
\right) 
\]
over $U\cap V$.
By Lemma~\ref{lem:isom-implies-lin-equiv}, we conclude that
$W|_{U\cap V} \sim W'|_{U\cap V}$ holds in $U\cap V$.
Write 
\[
W|_{U\cap V}-W'_{U\cap V} = {\rm div}(f)|_{U\cap V}
\]
for some $f\in \kk(X)$.
We can replace $W'$ with $W'+{\rm div}(f)$. 
By Lemma~\ref{lem:replacement-lin-equiv}, this replacement
preserves the equivariant isomorphism~\eqref{eq:isom-V}.
Thus, we may assume that 
$W|_{U\cap V} = W'|_{U\cap V}$.
Hence, we can find a Weil divisor 
$W''$ on $U\cup V$ for which there
is a $\mathbb{G}_m$-equivariant isomorphism 
\[
\pi^{-1}(U\cup V) \simeq {\rm Spec}_X\left( 
\bigoplus_{m \in \zz} \mathcal{O}_{U\cup V}(W'')
\right).
\]
This contradicts the maximality of $U$.
We conclude that $U=X$.
Thus, the statement of the theorem holds for $k=1$. 

Now, let $Y\rightarrow X$ be a 
$\mathbb{T}$-quasi-torsor. 
Let $Y\rightarrow Y_0$ be the quotient 
by a sub-torus $\mathbb{T}_0\leqslant \mathbb{T}$ of dimension $k-1$.
By Lemma~\ref{lem:quot-qetale-torus-cover}, 
we conclude that
$\pi_1\colon Y\rightarrow Y_0$ is a 
a $\mathbb{T}_0$-quasi-torsor
and $\pi_0\colon Y_0\rightarrow X$
is a $\mathbb{G}_m$-quasi-torsor. 
By induction on the dimension, we can find 
Weil divisors
$W_1,\dots, W_{k-1}$ on $Y_0$
for which
\[
Y \simeq {\rm Spec}_{Y_0}\left( 
\bigoplus_{(m_1,\dots,m_{k-1})\in \zz^{k-1}}
\mathcal{O}_{Y_0}(m_1W_1 +\dots+m_{k-1}W_{k-1})
\right) 
\] 
and $W$ on $X$ for which 
\[
Y_0 \simeq {\rm Spec}_{X}\left( 
\bigoplus_{m\in \zz}\mathcal{O}_X(mW)
\right).
\] 
Both isomorphisms are torus equivariant.
By Lemma~\ref{lem:torus-inv-lin-equiv-div}
and Lemma~\ref{lem:replacement-lin-equiv},
we may assume that each 
$W_i$, with $i\in\{1,\dots,k-1\}$
is torus invariant.
Since $Y_0\rightarrow X$ contains
no horizontal $\mathbb{G}_m$-invariant divisors, 
we conclude that 
$W_i=\pi_0^* W_{i,X}$
for some Weil divisors $W_{i,X}$.
Here, the pull-back is defined by restricting to the smooth locus.
We set
\[
Y' := {\rm Spec}_X\left( 
\bigoplus_{(m_1,\dots,m_k)\in\zz^k}
(m_1W_{X,1}+\dots+m_{k-1}W_{X,k-1}+m_kW_k)
\right). 
\]
Note that $Y'$ has a $\mathbb{T}_0$-quotient
$Y'_0$ obtained by considering the graded subring
given by $m_i=0$ for every $i\in \{1,\dots,k-1\}$.
This quotient is isomorphic to $Y_0$. 
Hence, we have a commutative diagram 
\[
 \xymatrix@R=2em@C=2em{
Y\ar[d]_-{\pi_1} & Y'\ar[d]^-{\pi_1'} \\
Y_0\ar[d]_-{\pi_0}\ar[r]^-{\phi} & Y_0'\ar[dl]^-{\pi_0'} \\
X.
 }
\] 
By construction, we have that 
\[
W_i = \pi_0^* W_{X,i} = \phi^* {\pi_0'}^* W_{X,i},
\]
holds for every $i\in \{1,\dots,k-1\}$.
We conclude that $Y'$ is $\mathbb{T}$-equivariantly isomorphic to $Y$.
This finishes the proof.
\end{proof}

\begin{proof}[Proof of Theorem~\ref{introthm-finite-torus-cover-klt-var}]
This statement is local.
Hence, it follows from
Theorem~\ref{introthm:torus-finite-cover}
and 
Theorem~\ref{introthm:cox-space-vs-g-covers}.
\end{proof}

In order to prove Theorem~\ref{introthm:iteration-torus-quasi-torsors},
we will need the following lemma. 

\begin{lemma}\label{lem:two-qt}
Let $\phi \colon Y \rightarrow X$ be a
$\mathbb{T}_\phi$-quasi-torsor 
and 
$\psi\colon Z\rightarrow Y$ 
be a $\mathbb{T}_\psi$-quasi-torsor.
The following statements hold:
\begin{enumerate}
    \item the composition
    $\phi\circ \psi\colon  Z\rightarrow X$ is a torus quasi-torsor,
    \item if $\phi$ corresponds to the subgroup
    $N_Y\leqslant {\rm WDiv}(X)$ and $Z\rightarrow X$ corresponds to the subgroup $N_Z\leqslant {\rm WDiv}(X)$, then $N_Y\leqslant N_Z$, and 
    \item the torus quasi-torsor $\psi$ is a torsor
    if and only if for every closed point $x\in X$ the images 
    $N_Y\rightarrow {\rm Cl}(X_x)$ and $N_Z\rightarrow {\rm Cl}(X_x)$ agree. 
    In particular, if the images of $N_Y$ and $N_Z$ agree on
    \[{\rm WDiv}(X)/{\rm CaDiv}(X)\] then $\psi$ is a torsor.
\end{enumerate}
\end{lemma}

\begin{proof}
We start by showing that if $\phi \colon Y \to X$ and $\psi \colon Z \to Y$ are two torus quasi-torsors with acting tori $\mathbb{T}_\phi$ and $\mathbb{T}_\psi$ respectively, 
then $\phi \circ \psi \colon Z \to X$ is a torus quasi-torsor as well. 
By Theorem~\ref{introthm:cox-space-vs-g-covers}, $\psi$ corresponds to a relative Cox ring, 
i.e. to a sheaf of graded algebras as in~\cite[Section~4.2.3]{ADHL15}. 
Thus, we can lift the action of $\mathbb{T}_\phi$ on $Y$ to $Z$ by~\cite[Proposition~4.2.3.6]{ADHL15}.
The quotient by the action of $\mathbb{T}_\phi \times \mathbb{T}_\psi$ is $\phi \circ \psi$. This is a quasi-torsor, since $\phi$ and $\psi$ are so. 
Again by Theorem~\ref{introthm:cox-space-vs-g-covers}, $\phi \circ \psi$ corresponds to a relative Cox ring with respect to a subgroup $N \leqslant \WDiv(X)$. This shows $(1)$.
The previous construction also shows $(2)$.

For $(3)$, note that $\psi$ is a torsor if and only if 
every element of $\phi^*N_Z$ is Cartier in $Y$.
Since $N_Y\leqslant N_Z$, this happens if and only if
the image of $N_Z$ equals the image of $N_Y$ in every local Class group of $X$.
\end{proof}

\begin{proof}[Proof of Theorem~\ref{introthm:iteration-torus-quasi-torsors}]
Let $X$ be a variety with klt type singularities.
Consider a sequence of morphisms:
\[
 \xymatrix@R=2em@C=2em{
X=X_0 &
X_1\ar[l]_-{\phi_1} &
X_2\ar[l]_-{\phi_2} &
X_3\ar[l]_-{\phi_3} & 
\dots \ar[l]_-{\phi_4} &
X_i\ar[l]_-{\phi_i} &
X_{i+1}\ar[l]_-{\phi_{i+1}}& 
\dots \ar[l]_-{\phi_{i+2}}&
}
\] 
such that each $\phi_i\colon X_i\rightarrow X_{i-1}$
is a $\mathbb{T}$-quasi-torsor.
We write $\psi_i=\phi_i\circ\dots\circ \phi_1$.
By Lemma~\ref{lem:two-qt}.(1), we know that each 
$\psi_i$ is a $\mathbb{T}$-quasi-torsor
corresponding to a subgroup $N_i\leqslant {\rm WDiv}(X)$.
By Lemma~\ref{lem:two-qt}.(2), we know that there is a sequence of subgroups
\[
N_1 \leqslant N_2 \leqslant \dots \leqslant N_i \leqslant \dots 
\] 
By Theorem~\ref{thm:wdiv-cdiv-klt-type}, we know that ${\rm WDiv}(X)/{\rm CaDiv}(X)$ is a finitely generated abelian group. 
In particular, for some $i_0$, we have that the image
of every $N_i$, with $i\geq i_0$, stabilizes in 
${\rm WDiv}(X)/{\rm CaDiv}(X)$. 
By Lemma~\ref{lem:two-qt}.(3), we conclude that each 
$\phi_i$, with $i\geq i_0$ is a torus torsor.
This finishes the proof of the theorem.
\end{proof} 

\section{Iteration of torus and finite covers} 

In this section, we study the iteration
of quasi-\'etale $G$-covers, where $G$ is either finite
or a torus.

\begin{proof}[Proof of Theorem~\ref{thm:iteration1}]
We consider a sequence
\[
 \xymatrix@R=2em@C=2em{
X=X_0 &
X_1\ar[l]_-{\phi_1} &
X_2\ar[l]_-{\phi_2} &
X_3\ar[l]_-{\phi_3} & 
\dots \ar[l]_-{\phi_4} &
X_i\ar[l]_-{\phi_i} &
X_{i+1}\ar[l]_-{\phi_{i+1}}& 
\dots \ar[l]_-{\phi_{i+2}}&
}
\]
as in the statement of  Theorem~\ref{thm:iteration1}. We claim that if $\phi_i$ is a $\mathbb{T}$-cover and $\phi_{i+1}$ a finite $G_{i+1}$-cover (for any $i\geq 1$), then we have a variety $X_{i}'$, a commutative diagram  
\[
 \xymatrix@R=2em@C=2em{
 X_{i+1} \ar[r]^-{\phi_{i+1}} \ar[d]_-{\phi_{i+1}'} & X_i \ar[d]^-{\phi_i} \\
   X_{i}' \ar[r]^-{\phi_{i}'} & X_{i-1} ,
 }
\]
where $\phi_{i}'$ (resp. $\phi_{i+1}'$) is a $G_{i+1}$-cover (resp. $\mathbb{T}$-cover), which is \'etale if and only if $\phi_{i+1}$ (resp. $\phi_{i}$) is \'etale. This claim holds since by~\cite[Proof of Proposition 5.1]{BM21}, we have an exact sequence
\[
\xymatrix@C=15pt{
\zz^{\dim(\mathbb{T})} \ar[r] & \pi_1(X_i^{\rm reg}) \ar[r]^{\varphi} & \pi_1(X_{i-1}^{\rm reg}) \ar[r] &
1,
}
\]
where $\mathbb{T}$ is the general fiber of $\phi_i$. Thus if the finite cover $\phi_{i+1}$ corresponds to the normal subgroup $N \leqslant \pi_1(X_{i}^{\mathrm{reg}})$, then we get $\phi_{i}'$ as the finite cover of $X_{i-1}$ corresponding to the image of $N$ in $\pi_1(X_i^{\mathrm{reg}})$ under the above homomorphism $\varphi$. By the fiber product $X_{i+1}=X_i \times_{X_{i-1}} X_i'$ (which preserves \'etaleness, finiteness and GIT-quotients), we get the commutative diagram. 

Thus, if for all $j \in \mathbb{N}$, there exists a $k\geq j$ such that $\phi_k$ is a finite quasi-\'etale but not \'etale cover, by reordering of the $\phi_i$ according to the claim just proven, we can construct an infinite sequence
\[
 \xymatrix@R=2em@C=2em{
X'=X_0 &
X'_1\ar[l]_-{\phi'_1} &
X'_2\ar[l]_-{\phi'_2} &
X'_3\ar[l]_-{\phi'_3} & 
\dots \ar[l]_-{\phi'_4} &
X'_i\ar[l]_-{\phi'_i} &
X'_{i+1}\ar[l]_-{\phi'_{i+1}}& 
\dots \ar[l]_-{\phi'_{i+2}}&
}
\]
where the $\phi'_i$ are finite Galois quasi-\'etale but not \'etale covers. But this is a contradiction to~\cite[Theorem 1.1]{GKP16}. 
Thus, there are only finitely many finite Galois quasi-\'etale and not \'etale covers in the sequence, i.e., there are only finitely many finite quasi-torsors in this sequence that are not finite torsors.
\end{proof}

In what follows, we turn to prove Theorem~\ref{thm:iteration2}.
To do so, we will use the following lemmata.

\begin{lemma}\label{lem:square-diagram}
Let $\phi\colon Y\rightarrow X$ be a $\mathbb{T}$-quasi-torsor of a klt type variety corresponding to the subgroup
$N$ of ${\rm WDiv}(X)$.
Let $\pi\colon X'\rightarrow X$ be a finite torsor.
Let $Y'=Y\times_X X'$ and $Y'\rightarrow X'$ be the associated $\mathbb{T}$-quasi-torsors so that we have a commutative diagram:
\[
 \xymatrix@R=2em@C=2em{
 Y \ar[d]_-{\phi} & Y' \ar[l]\ar[d]^-{\phi'} \\
 X & X'\ar[l]_-{\pi}. 
 }
\]
Then, $\phi'$ is the $\mathbb{T}$-quasi-torsor associated to the subgroup $\pi^*N$ of ${\rm WDiv}(X')$.
\end{lemma} 

\begin{proof}
We consider the dual diagram
\[
 \xymatrix@R=2em@L=2em{
 \bigoplus_{D \in N} \mathcal{O}_X(D) \ar[r]& \left(\bigoplus_{D \in N} \mathcal{O}_X(D)\right) \otimes_{\mathcal{O}_X} \mathcal{O}_{X'} \\
 \mathcal{O}_X\ar[r] \ar[u] & \mathcal{O}_{X'}. \ar[u]
 }
\]
The top right entry equals $\bigoplus_{D \in N} \left(\mathcal{O}_X(D) \otimes_{\mathcal{O}_X} \mathcal{O}_{X'}\right)$. We have $\mathcal{O}_{X'}(\pi^*D)=\mathcal{O}_X(D) \otimes_{\mathcal{O}_X} \mathcal{O}_{X'}$ by definition of the pullback. So the claim follows.
\end{proof}

\begin{lemma}\label{lem:dominating-qt}
Let $Y\rightarrow X$ be a torus quasi-torsor corresponding to
the subgroup $N_Y$ of ${\rm WDiv}(X)$.
Let $Z\rightarrow X$ be a torus quasi-torsor corresponding to the subgroup $N_Z$ of ${\rm WDiv}(X)$.
If $N_Z\geqslant N_Y$ and $N_Z/N_Y$ is torsion free, then 
there is an induced 
torus quasi-torsor $Z\rightarrow Y$.
\end{lemma}

\begin{proof}
The condition that $N_Z/N_Y$ is torsion free means that we have a direct product representation $N_Z=N_Y \oplus N'$ with a subgroup $N'$ of $N_Z$ isomorphic to $N_Z/N_Y$. The downgrading of $\mathcal{O}_Z$ from $N_Z$ to $N'$ gives an action of a subtorus $\mathbb{T}_{N'} \leqslant \mathbb{T}_{N_Z}$ on $Z$. By construction, $Z/\!\!/\mathbb{T}_{N'}=Y$. Now the statement follows from Lemma~\ref{lem:quot-qetale-torus-cover}.
\end{proof}

\begin{proof}[Proof of Theorem~\ref{thm:iteration2}]
The proof of the theorem will consist of three steps. We briefly explain the steps here. In the first step, we will produce a singular variety with a special toric divisor.
In the second step, we will produce an infinite sequence of finite torsors for such a singular variety. We show that the rank of the group of Weil divisors modulo Cartier divisors diverges in this sequence.
Finally, we will use this divergence property to produce the infinite sequence of $\mathbb{T}$-quasi-torsors that are not torsors.\\

\noindent\textit{Step 1:} For each $n\geq 2$, we construct a $n$-dimensional projective variety with a single isolated toric singularity and infinite \'etale fundamental group.\\

Let $Z^n$ be a smooth projective variety with infinite \'etale fundamental group.
Let $z\in Z^n$ be a smooth point.
In local coordinates around $z\in Z$,
the formal completion $\hat{\mathcal{O}}_{Z,z}$ corresponds to the standard fan 
$\langle e_1,\dots, e_n \rangle \subset \rr^n$.
For each $n\geq 2$, we consider the blow-up given by the fan decomposition
\[
\Sigma_n:=\{ 
\langle \bar{e_1}, e_2, e_3,\dots, e_n, v\rangle,
\langle e_1, \bar{e_2},e_2,e_3, \dots, v\rangle, \dots ,\langle e_1,\dots, e_{n-1},\bar{e_n},v\rangle 
\},
\]
where $v=2e_1+e_2+e_3+\dots+e_n$.
We let $Y^n \rightarrow Z^n$ to be the corresponding blow-up. 
Observe that $Y^n$ has a unique isolated toric singularity.
We let $y^n\in Y^n$ be such isolated toric singularity.
Note that the local Class group
of $Y^n$ at $y^n$ is 
$\zz_2$.
For $n=2$, this point is a rational double point.
By construction, there is a divisor $T^n\subset Y^n$ which is a normal projective toric variety 
and $y^n$ is contained in $T^n$.
Indeed, this toric variety corresponds to the primitive lattice generator $v\in \Sigma_n(1)$.
Since $Y^n$ has klt singularities and $Z^n$ is smooth, we conclude that
$\pi_1(Y^n)\simeq \pi_1(Z^n)$.
In particular, the \'etale fundamental group of $Y^n$ is infinite.\\

\noindent\textit{Step 2:} We construct a sequence of finite \'etale Galois covers of $Y^n$ 
and study their groups of Weil divisors modulo Cartier divisors.\\

Let 
\[
 \xymatrix@R=2em@C=2em{
Y^n=Y^n_0 &
Y^n_1\ar[l]_-{f_1} &
Y^n_2\ar[l]_-{f_2} &
Y^n_3\ar[l]_-{f_3} & 
\dots \ar[l]_-{f_4} &
Y^n_i\ar[l]_-{f_i} &
Y^n_{i+1}\ar[l]_-{f_{i+1}}& 
\dots \ar[l]_-{f_{i+2}}&
}
\]
be an infinite sequence of finite \'etale Galois covers. 
Let $k_i:={\rm deg}(Y^n_{i+1}\rightarrow Y^n_i)$
be the degree of the cover.
Then, the variety $Y^n_i$ is $n$-dimensional and it has $k_0\cdots k_{i-1}$ isolated singularities.
We denote these singularities as 
\[
y^n_{i,(m_0,\dots,m_{i-1})} \in Y^n_i,
\]
where $1\leq m_j \leq k_j$ for each $j$.
We can order the singularities in such a way that
\[
f_i^{-1}(y_{i-1,(m_0,\dots,m_{i-1})})
= 
\bigcup_{m=1}^{k_i} y^n_{i,(m_0,\dots,m_{i-1},m)}.
\]
Since $T^n$ has trivial fundamental group, we conclude that 
$f_i^*\dots f_1^*T^n$ is the disjoint union 
of $k_0\cdots k_{i-1}$ toric varieties
isomorphic to $T^n$.
We write
$T^n_{i,(m_0,\dots,m_{i-1})}$ 
with $1\leq m_j\leq k_j$ for such toric divisors.
By construction,
the toric divisor
$T^n_{i,(m_0,\dots,m_{i-1})}$ contains the singular point 
$y^n_{i,(m_0,\dots,m_{i-1})}$.
We claim that 
\begin{equation}\label{eq:isom-weil-mod-cart} 
{\rm WDiv}(Y_i^n)/{\rm CaDiv}(Y_i^n) \simeq \bigoplus_{i=1}^{k_0\cdots k_{i-1}} \zz_2.
\end{equation} 
First, observe that $T^n_{i,(m_0,\dots,m_{i-1})}$ is not a Cartier divisor.
Indeed, if this was the case then
$T^n_{i,(m_0,\dots,m_{i-1})}$ would be analytically Cartier around
$y^n_{i,(m_0,\dots,m_{i-1})}$.
This implies that $T^n$ is analytically Cartier around $y^n$, leading to a contradiction.
On the other hand
$2T^n_{i,(m_0,\dots,m_{i-1})}$ is Cartier in $Y^n_i$ as
it is the pull-back of $2T^n$ on a neighborhood of the only singular point that it contains.
We conclude that each 
$T^n_{i,(m_0,\dots,m_{i-1})}$ is $2$-torsion in the abelian group
${\rm WDiv}(Y^n)/{\rm CaDiv}(Y^n)$.
Let $J\subset ([1,k_0] \cap \zz) \times\dots\times ([1,k_{i-1}]\cap \zz)$ be a subset. 
Assume that we have a relation of the form 
\[
\sum_{j\in J} T^n_{i,j} =0 \in {\rm WDiv(Y^n)}/{\rm CaDiv}(Y^n).
\]
This means that the divisor
$\sum_{j\in J} T^n_{i,j}$ is Cartier in $Y^n_i$.
Let $j_0\in J$ be a fixed element.
For each $j_k \neq j_0$ in $J$, we have that $T^n_{i,j_k}$ is Cartier at $y^n_{i,j_0}$. 
We conclude that $T^n_{i,j_0}$ is Cartier at $y^n_{i,j_0}$.
Hence, it is a Cartier divisor. This leads to a contradiction. 
Then, the isomorphism~\eqref{eq:isom-weil-mod-cart} holds.\\

\noindent\textit{Step 3:} In this step, we construct an infinite sequence of finite torsors and
torus quasi-torsors of $Y^n$.\\

For each $i\geq 1$, we denote by $N_i$
the group of ${\rm WDiv}(Y^n_i)$ generated by
\[
\{ T^n_{i,(m_0,\dots,m_{i-1})} \mid 
1\leq m_0\leq k_0, \dots, 
1\leq m_{i-2}\leq k_{i-2}, 
\text{ and }
1\leq m_{i-1} \leq k_{i-1}-1
\}.
\]
For each $i\geq 0$, 
we define $X_{2i}^n$ to be the relative Cox ring of $Y_i^n$ with respect to $N_i$.
We define $X_{2i+1}$
to be $X_{2i}\times_{Y_i^n} Y_{i+1}^n$.
We define $\phi_{2i+1}\colon X_{2i+1}^n \rightarrow X_{2i}^n$ to be the induced morphism.
Thus, we have a commutative diagram as follows:
\[
 \xymatrix@R=2em@C=2em{
 X_{2i}^n \ar[d] & X_{2i+1}^n \ar[l]_-{\phi_{2i+1}} \ar[d] \\
 Y^n_i & Y^n_{i+1}\ar[l]_-{f_{i+1}} ,
 }
\]
By Lemma~\ref{lem:square-diagram},
the torus quasi-torsor
$X_{2i+1}^n \rightarrow Y^n_{i+1}$ is induced by the subgroup
$f_{i+1}^* N_i \leqslant {\rm WDiv}(Y^n_{i+1})$.
By construction, we have that
$f_{i+1}^*N_i \leqslant N_{i+1}$.
By Lemma~\ref{lem:dominating-qt} there is a corresponding quasi-torsor $\phi_{2(i+1)} \colon X_{2(i+1)}\rightarrow X_{2i+1}^n$.

We claim that $\phi_{2(i+1)}$ is not a torus torsor.
Let $C$ be the Class group
of $Y_{i+1}$ at the point
$y_{i+1,(k_0,\dots,k_{i-1},1)}$.
By the isomorphism~\eqref{eq:isom-weil-mod-cart},
we know that $C\simeq \zz_2$.
Note that the image of $N_{i+1}$ in $C$ is isomorphic to $\zz_2$.
Indeed, the image of the divisor $T^n_{i+1,(k_0,\dots,k_{i-1},1)}$ generates $C$.
On the other hand,
the image of $f_{i+1}^*N_i$ in $C$ is trivial
since no divisor among the generators of $N_i$ 
pass through $y_{i,(k_0,\dots,k_{i-1})}$.
By Lemma~\ref{lem:two-qt}, we conclude that $\phi_{2(i+1)}$ is a torus quasi-torsor which is not a torsor.
We deduce that there exists an infinite sequence 
\[
 \xymatrix@R=2em@C=2em{
X^n=X^n_0 &
X^n_1\ar[l]_-{\phi_1} &
X^n_2\ar[l]_-{\phi_2} &
X^n_3\ar[l]_-{\phi_3} & 
\dots \ar[l]_-{\phi_4} &
X^n_i\ar[l]_-{\phi_i} &
X^n_{i+1}\ar[l]_-{\phi_{i+1}}& 
\dots \ar[l]_-{i+2}&
}
\]
satisfying the following conditions:
\begin{itemize}
\item $X^n$ is a $n$-dimensional projective variety with a single isolated toric singularity,
\item each $\phi_i$, with $i$ odd, is a finite torsor, and
\item each $\phi_i$, with $i\geq 2$ even, is a $\mathbb{T}$-quasi-torsor which is not a torsor.
\end{itemize}
This finishes the proof.
\end{proof}

Now, we turn to prove Theorem~\ref{thm:iteration3}.
We will need the following two lemmata.

\begin{lemma}\label{lem:pullback-cl}
Let $f\colon X'\rightarrow X$ be a finite $G$-torsor.
Let $x\in X$ and $x'\in f^{-1}(x)$ be two closed points.
Then, the induced homomorphism
$f^*\colon {\rm Cl}(X;x)\rightarrow {\rm Cl}(X';{x'})$
is a monomorphism.
\end{lemma}

\begin{proof}
Let $W$ be a Weil divisor through $x \in X$ such that $f^*W$ is principal near $x'$. We can even assume that there is a $G$-invariant open around $x'$ where $f^*W=\mathrm{div}(h)$. By~\cite[Thm.~II.3.1]{AGIV}, $h$ is a semiinvariant, i.e. $g^*h=\chi(g) h$, where $\chi(g) \in \mathbb{K}^*$, for every $g$ in $G$. But the induced action via $\chi \colon G \to \mathbb{K}^* \subseteq \mathbb{K}[X]^*$ on $X$ is ramified if it is nontrivial. This can not happen since $f$ is \'etale. Therefore, $h$ is $G$-invariant and defines $W=\mathrm{div}_{U}(h)$ on some $x \in U \subseteq X$. The claim follows.
\end{proof}

\begin{lemma}\label{lem:factorial-qt}
Let $Y\rightarrow X$ be the $\mathbb{T}$-quasi-torsor
associated to the group $N\leqslant {\rm WDiv}(X)$.
The variety $Y$ is locally factorial if and only if
$N\rightarrow {\rm Cl}(X;x)$ is surjective for every
closed point $x\in X$.
In particular, if $N\rightarrow {\rm WDiv}(X)/{\rm CaDiv}(X)$ is surjective, then $Y\rightarrow X$ is a factorial $\mathbb{T}$-quasi-torsor. 
\end{lemma}

\begin{proof}
The statement follows from~\cite[Theorem~1.3.3.3]{ADHL15} applied to the spectrum $X_x$ of the local ring $\mathcal{O}_{X,x}$, where we view $N$ as a subgroup of ${\rm WDiv}(X_x)$ by restriction. Then the aforementioned theorem says that the stalk $\mathcal{R}(X)_{N,x}$ is factorial if and only if $N\rightarrow {\rm Cl}(X;x)$ is surjective. In fact,~\cite[Theorem~1.3.3.3]{ADHL15} only states one direction, but it directly follows from applying~\cite[Theorem~1.3.3.1]{ADHL15} to the smooth locus, which gives an equivalence. It is clear that $Y$ is locally factorial if and only if all the stalks $\mathcal{R}(X)_{N,x}$ are factorial.
\end{proof}

\begin{proof}[Proof of Theorem~\ref{thm:iteration3}]
Consider a sequence 
\[
 \xymatrix@R=2em@C=2em{
X=X_0 &
X_1\ar[l]_-{\phi_1} &
X_2\ar[l]_-{\phi_2} &
X_3\ar[l]_-{\phi_3} & 
\dots \ar[l]_-{\phi_4} &
X_i\ar[l]_-{\phi_i} &
X_{i+1}\ar[l]_-{\phi_{i+1}}& 
\dots \ar[l]_-{\phi_{i+2}}&
}
\]
as in the statement of the theorem.
This means that every $\phi_i$ is either a factorial
$\mathbb{T}$-quasi-torsor or a finite quasi-torsor.
By Theorem~\ref{thm:iteration1}, we may assume, after possibly truncating our sequence, that every finite quasi-torsor in this sequence is a finite torsor, i.e., a finite Galois \'etale cover.
Proceeding as in the proof of Theorem~\ref{thm:iteration1}, we obtain a sequence of finite torsors:
\[
 \xymatrix@R=2em@C=2em{
X'=X'_0 &
X'_1\ar[l]_-{\phi'_1} &
X'_2\ar[l]_-{\phi'_2} &
X'_3\ar[l]_-{\phi'_3} & 
\dots \ar[l]_-{\phi'_4} &
X'_i\ar[l]_-{\phi'_i} &
X'_{i+1}\ar[l]_-{\phi'_{i+1}}& 
\dots \ar[l]_-{\phi'_{i+2}}&
}
\]
so that each $X_i\rightarrow X'_i$ is a relative Cox ring
with respect to the subgroup
$N_i\leqslant {\rm WDiv}(X_i')$.
By Lemma~\ref{lem:two-qt}, we know that 
every $\mathbb{T}$-quasi-torsor over a factorial variety is
indeed a torsor.
We write $\psi_i = \phi_i \circ \dots \circ \phi_1$.
By~\cite[Theorem 3.4.(1)]{bf84}, there exists a locally closed decomposition 
$X'=\bigsqcup_{j\in J} Y'_j$
such that the class group
of $X'_i$ at $x$ is independent of $x\in \psi_i^{-1}(Y'_j)$.
Applying Lemma~\ref{lem:pullback-cl} to closed points of $Y'_j$, we conclude that there exists $i_0\in \zz_{>0}$, 
such that 
\begin{equation}\label{eq:isoms-cl} 
{\phi'_i}^*\colon {\rm Cl}({X'_{i}};y)\rightarrow 
{\rm Cl}({X'_{i+1}};x)
\end{equation}  
is an isomorphism for every
$x\in X'_i$,
every $y\in {\phi'_i}^{-1}(x)$, and $i\geq i_0$.
It suffices to show that whenever $X_i$ is factorial
and $X_{i+1}\rightarrow X_i$ is a finite \'etale Galois cover in our sequence, the variety $X_{i+1}$ is again factorial for every $i\geq i_0$. 
In this case, we have a commutative diagram 
\[
 \xymatrix@R=2em@C=2em{
 X_i \ar[d]_-{\pi_i} & X_{i+1}\ar[d]^-{\pi_{i+1}} \ar[l]_-{\phi_{i}} \\
 X_i' & X_{i+1}\ar[l]_-{\phi_i'}
 }
\]
where $\pi_i$ is the relative Cox ring over $X'_i$ with respect to the subgroup $N_i\leqslant {\rm WDiv}(X_i')$.
By Lemma~\ref{lem:square-diagram}, the $\mathbb{T}$-quasi-torsor
$X_{i+1}\rightarrow X_i$ is induced by
${\phi_i'}^*N_i \leqslant {\rm WDiv}(X_{i+1})$.
By Lemma~\ref{lem:factorial-qt}, we know that for
every closed point $x\in X_i'$, the induced homomorphism
\begin{equation}\label{eq:surj-local-class}
N_i \rightarrow {\rm Cl}({X_i'};x)
\end{equation} 
is surjective.
By isomorphism~\eqref{eq:isoms-cl} and surjectivity~\eqref{eq:surj-local-class}, we conclude that for every closed point $x\in X'_{i+1}$ we have that 
${\phi'_i}^*N_i\rightarrow {\rm Cl}(X'_{i+1};x)$ is surjective. By Lemma~\ref{lem:factorial-qt}, we conclude that $X_{i+1}$ is a factorial variety. 
This finishes the proof.
\end{proof} 

\begin{proof}[Proof of Theorem~\ref{introthm:factorial-cover}]
Due to Theorem~\ref{thm:iteration3}, we may find a variety $Y$ 
satisfying the following properties:
\begin{enumerate}
    \item[(i)] every finite quasi-torsor over $Y$ is a finite torsor, 
    \item[(ii)] every factorial $\mathbb{T}$-quasi-torsor over a finite torsor of $Y$ is a $\mathbb{T}$-torsor, 
    \item[(iii)] $Y$ admits the action of a reductive group $G$,
    \item[(iv)] the group $G$ is an extension of an algebraic torus by a finite solvable group, and 
    \item[(v)] the isomorphism $X\simeq Y/\!\!/G$ holds.
\end{enumerate}
Note that condition (i) implies that the natural epimorphism
\[
\hat{\pi}_1(Y^{\rm reg})\rightarrow
\hat{\pi}_1(Y)
\]
is an isomorphism.
Otherwise, we could find a finite Galois quasi-\'etale cover of $Y$ that ramifies over the singular locus.
This shows that (1) in the statement of the theorem holds.

Let $Y'\rightarrow Y$ be a finite quasi-\'etale morphism.
By condition (i) this morphism is indeed
a finite \'etale morphism. 
Assume that $Y'$ is not factorial at the point $y'$.
By~\cite[Theorem 3.7]{GKP16}
there exists a finite \'etale Galois morphism
$Y''\rightarrow Y$ 
such that $Y''$ admits a finite \'etale Galois morphism to $Y'$.
By Lemma~\ref{lem:pullback-cl}, we conclude that $Y''$ is not factorial.
Thus, $Y''$ admits a factorial $\mathbb{T}$-quasi-torsor that is not a $\mathbb{T}$-torsor.
Indeed, we can take the relative Cox ring of $Y''$ with respect to a subgroup $N$ of ${\rm WDiv}(Y'')$
that surjects onto
${\rm WDiv}(Y'')/{\rm CaDiv}(Y'')$.
This contradicts condition (ii).
We conclude that (2) in the statement of the theorem holds. 
Note that (iii)-(v) are the same than
(3)-(5) in the statement of the theorem.
This finishes the proof.
\end{proof}

\begin{proof}[Proof of Theorem~\ref{introthm:toric-quot}]
Locally toric singularities are klt type singularities.
Then, we can apply Theorem~\ref{thm:iteration1} to deduce that 
$X$ admits a torus quasi-torsor $Y$ that is factorial.
By Theorem~\ref{introthm-finite-torus-cover-klt-var}, the variety $Y$ has klt type singularities, hence canonical factorial singularities.
The local Cox ring of a locally toric singularity is a locally toric singularity.
Hence, the variety $Y$ has factorial locally toric singularities.
However, a factorial toric singularity is smooth.
We conclude that $Y$ is a smooth variety.
\end{proof}

\subsection{Normal singularities}
\label{subsec:normal-singularities}
In this subsection, we show that some of the proofs explained above naturally generalize to normal singularities with some minor considerations.

\begin{proof}[Proof of Theorem~\ref{introthm:optimal-T}]
If each Class group ${\rm Cl}(X;x)$ is finitely generated, then by Theorem~\ref{thm:wdiv-cdiv}, we know that
${\rm WDiv}(X)/{\rm CaDiv}(X)$ is finitely generated.
Then, the proof is verbatim from the proof of Theorem~\ref{introthm:iteration-torus-quasi-torsors}.
This shows that $(2)$ implies $(1)$.

Now, we turn to prove that $(1)$ implies $(2)$.
On the other hand assume that some local Class group ${\rm Cl}(X;x_0)$ is not finitely generated. We consider an infinite sequence of divisors
$\{W_i\}_{i\in \mathbb{N}}$ in ${\rm WDiv}(X)$
such that the image of $N_k:=\langle W_1,\dots,W_k\rangle$ in ${\rm Cl}(X;x_0)$ strictly contains $N_{k-1}$.
Let $X_i\rightarrow X$ be the $\mathbb{T}$-quasi-torsor associated to $N_i$. 
By construction, the quotients 
$N_{i+1}/N_i$ are free.
Then, by Lemma~\ref{lem:dominating-qt}, we have associated 
torus quasi-torsors 
\[
 \xymatrix@R=2em@C=2em{
X=X_0 &
X_1\ar[l]_-{\phi_1} &
X_2\ar[l]_-{\phi_2} &
X_3\ar[l]_-{\phi_3} & 
\dots \ar[l]_-{\phi_4} &
X_i\ar[l]_-{\phi_i} &
X_{i+1}\ar[l]_-{\phi_{i+1}}& 
\dots \ar[l]_-{\phi_{i+2}}&
}
\] 
By Lemma~\ref{lem:two-qt}.(3), no 
$\phi_i$ is a $\mathbb{T}$-torsor.
\end{proof}

Now, we turn to give a proof of Theorem~\ref{introthm:optimal-mixed} that discusses optimal normal singularities for which sequences of quasi-torsors are eventually torsors.
We will use the following lemma.

\begin{lemma}\label{lem:rat1-sing-cover}
Let $f\colon X'\rightarrow X$ be a quasi-\'etale finite Galois morphism. Let $x\in X$ and $x'\in f^{-1}(x)$ be finite points.
Then, if ${\rm Cl}(X';x')$ is finitely generated, then 
${\rm Cl}(X;x)$ is finitely generated.
\end{lemma}

\begin{proof}
Let $W$ be a Weil divisor through $x\in X$ such that $f^*W$ is Cartier.
Then, passing to a $G$-invariant affine we may assume $f^*W={\rm div}(h)$. The regular function $h^{|G|}$ is $G$-invariant. 
Hence $mW\sim 0$ around $x$.
We conclude that the kernel of
$f^*\colon {\rm Cl}(X;x)\rightarrow {\rm Cl}(X';x')$ is torsion.
So, if ${\rm Cl}(X;x)$ is not finitely generated, then 
$f^*{\rm Cl}(X;x)$ is not finitely generated and the claim follows.
\end{proof}

\begin{proof}[Proof of Theorem~\ref{introthm:optimal-mixed}]
First, assume that conditions $(a)$ and $(b)$ are satisfied. 
In the proof of Theorem~\ref{thm:iteration3}, we used the argument of Theorem~\ref{thm:iteration1}
to deduce that the finite quasi-torsors are eventually torsors. 
The same argument 
goes through in the present case by 
replacing~\cite[Theorem 1.1]{GKP16}
with~\cite[Theorem 1]{Sti17}.

On the other hand, in the proof of Theorem~\ref{thm:iteration3}, we used
the constructibility of the functor ${\rm Cl}$ in the \'etale topology.
We argue that this still holds in the present setting.
By Lemma~\ref{lem:rat1-sing-cover}, every local Class group ${\rm Cl}(X;x)$ is finitely generated.
Let $f\colon X'\rightarrow X$ be a resolution of singularities. 
In this case we have that
$R^1f_*(\mathcal{O}_{X'})=0$ 
as the local Class groups are finitely generated.
Then, we can apply~\cite[Theorem 3.1.(4)]{GKP16} to conclude that the Class group functor is constructible in the \'etale topology.
Now, the proof is verbatim from the proof of Theorem~\ref{thm:iteration3}.

Now, assume that condition $(1)$ is not satisfied.
By~\cite[Theorem 1]{Sti17}, there is an infinite sequence of finite quasi-torsors of $X$ that are not torsors.
On the other hand, assume that condition $(2)$ is not satisfied.
Then, there exists a finite quasi-\'etale cover $X'\rightarrow X$ and a point $x'$ for which ${\rm Cl}(X';x')$ is not finitely generated.
Hence, proceeding as in the proof of Theorem~\ref{introthm:optimal-T}, we conclude that there exists an infinite sequence of $\mathbb{T}$-quasi-torsors
over $X'$ that are not torsors.
\end{proof} 

\section{Examples and Questions}

In this section, we collect some examples related to the theorems of the article
and some questions that lead to further research.

\begin{example}\label{ex:cod-1-ram}
{\em 
In this example, we show that 
finite covers of klt singularities ramified
over codimension one points may not be klt.
Let $D=\{(y,z)\mid y^3+z^m=0 \}\subset \mathbb{A}^2_{y,z}$.
Then the singularity
\[
X:=\{ (x,y,z) | x^2+y^3+z^m=0\} 
\]
is a double cover of $\mathbb{A}^2_{y,z}$ ramified along $D$.
The singularity $(X;(0,0,0))$ is Du Val
for $m\in \{4,5\}$.
Otherwise, it is not a klt surface singularity.
}
\end{example}

\begin{example}\label{ex:torsion-quotient}
{\em 
In the construction of the local Cox ring of a singularity $(X;x)$, 
we choose a homomorphism
$N\rightarrow {\rm Cl}(X;x)$ with kernel $N_0$ 
and a character $\chi \colon N^0\rightarrow \mathbb{K}(X)^*$ for which
${\rm div}(\chi(E))=E$
for all $E\in N^0$.
If $X$ is a projective variety for which ${\rm WDiv}(X)/{\rm CaDiv}(X)$ is torsion, 
then we can consider a surjective homomorphism
$N\rightarrow {\rm WDiv}(X)/{\rm CaDiv}(X)$
with kernel $N^0$
and a character as before $\chi\colon N^0\rightarrow \mathbb{K}(X)^*$.
If $\mathcal{I}$ is the ideal sheaf 
generated by $1-\chi(E)$ with $E\in N^0$, 
then the morphism
\[
Y:={\rm Spec}_X(\mathcal{R}(X)_N/\mathcal{I})\rightarrow X
\]
is finite and ramifies over codimension one points. A priori, it is not clear how to control the divisors over which the previous morphism ramifies.
Hence, it is not clear whether $Y$ has klt type singularities provided that $X$ has klt type singularities.
This means that the concept of relative Cox ring with quotients induced by characters is not well-behaved from the singularities perspective.
}
\end{example}

\begin{example}\label{ex:sing-improve}
{\em 
The local Cox ring of a klt type singularity $(X;x)$ is non-trivial whenever
${\rm Cl}(X;x)$ is non-trivial.
If $(Y;y)\rightarrow (X;x)$ is the spectrum of the local Cox ring of $(X;x)$,
then we expect that the equations defining $(Y;y)$ are somewhat simpler than the equations defining $(X;x)$.
Although, whenever ${\rm Cl}(X;x)$ has non-trivial free part, the dimension of $(Y;y)$ is larger than the dimension of $(X;x)$.
For instance, if $(X;x)$ is a toric singularity, then
${\rm Cox}(X;x)$ is a smooth point of dimension $\dim X +\rho(X_x)$.
}
\end{example}

\begin{example}\label{ex:thm7-analytically-local}
{\em 
Consider the affine variety
\[
X=\{(x,y,z,w)\mid xy+zw+z^3+w^3=0\}.
\]
The variety $X$ is canonical with isolated singularities.
Furthermore, $X$ is factorial at $x:=(0,0,0,0)$.
However, $X$ is not analytically factorial at $x$. 
Let $Y\rightarrow X$ be a $\mathbb{T}$-quasi-torsor.
Then, $\pi\colon Y\rightarrow X$ is a $\mathbb{T}$-torsor on an affine neighborhood of $x\in X$, by Lemma~\ref{lem:two-qt}.
Hence, $Y$ is singular along $\pi^{-1}(x)$.
This leads to a contradiction.
}
\end{example}

\begin{example}
\label{ex:semisimple} 
{\em 
Over a smooth point, every finite quasi-torsor is a torsor
and every torus quasi-torsor is a torsor.
In this example, we show the existence of a ${\rm SL}_n(\kk)$-quasi-torsor over a smooth germ that is not a torsor.
We refer the reader to~\cite{Pop92} for more examples in this direction.

In what follows, we let $n\geq 2$.
Let $W$ be the space of linear transformations from 
$\cc^{n+1}$ to $\cc^n$.
Note that $W$ has dimension $n^2+n$.
Let ${\rm SL}_n(\cc)$ act on $W$ by acting on the range of the linear function.
The action is free exactly at all the points corresponding to surjective linear transformations.
The closure of the orbit of an element contains $0\in W$ if and only if the corresponding linear transformation does not have full rank.
Let $U\subset W$ be the open set consisting of surjective linear transformations.
Note that the ${\rm SL}_n(\cc)$-action naturally extend to a ${\rm Gl}_n(\cc)$-action.
The space of surjective linear transformations, up to the ${\rm GL}_n(\cc)$-action, is parametrized by their kernels.
Hence, the quotient $U/\!\!/{\rm GL}_n(\cc)\simeq \pp^{n}$.
Furthermore, we have that 
$U/\!\!/{\rm SL}_n(\cc)\simeq \mathbb{A}^{n+1}-\{0\}$.
It follows that 
$W/\!\!/{\rm SL}_n(\cc)\simeq \mathbb{A}^n$.
As explained above, the action is free on $U$
so $W\rightarrow \mathbb{A}^n$ is a ${\rm SL}_n(\cc)$-torsor over $\mathbb{A}^n-\{0\}$.
On the other hand, the fiber over $\{0\}$ is given by the vanishing of at least $n$ minors, so its codimension in $W$ is at least $2$.
Thus, $W\rightarrow \mathbb{A}^n$ gives a ${\rm SL}_n(\cc)$-quasi-torsor which is not a torsor.
}
\end{example}

\begin{remark}
{\em 
One of the reasons that makes the understanding of finite quasi-torsors  of a singularity $(X;x)$ easier than other kinds of quasi-torsors
is the existence of an algebraic object that detects them.
The same holds for torus quasi-torsors.

The previous example might point in the direction that this is not true anymore for arbitrary reductive groups, which might render the  study of their (quasi-)torsors a lot more complicated.
}
\end{remark}

\begin{question}
{\em 
In Theorem~\ref{introthrm:gl2-cover}, we showed that there exists a $3$-fold toric singularity $(T;t)$ that admits a $\mathbb{P}{\rm GL}_3(\kk)$-cover from a $5$-dimensional singularity which is not of klt type.
This cover is unramified over codimension points over $T$ so the pathology in Example~\ref{ex:cod-1-ram} does not happen.
This naturally leads to the following question.
Is there a $G$-cover over a surface klt singularity that is unramified over codimension one points and is not of klt type?
}
\end{question}

\begin{question}\label{quest:G-qt-klt-type}
{\em 
In Theorem~\ref{introthm:torus-finite-cover}, we showed that if
$(X;x)$ is a klt type singularity
and $G$ is a finite extension of a torus, then a $G$-quasi-torsor over $(X;x)$ is of klt type. 
Does this statement still hold if we only assume that $G$ is a reductive group?
We expect that there are counter-examples for this statement if $G$ is a unipotent group.
}
\end{question}

\bibliographystyle{habbrv}
\bibliography{mybibfile}

\end{document}